\documentclass[a4paper, 11pt, reqno]{amsart}
\usepackage{mathtools}
\mathtoolsset{showonlyrefs}

\usepackage{fullpage}
\usepackage[utf8]{inputenc}
\usepackage[english]{babel}
\usepackage{amsmath}
\usepackage{amssymb}
\usepackage{amsthm}
\usepackage{xifthen}
\usepackage{holtpolt}
\usepackage{xspace}
\usepackage[usenames]{xcolor}
\usepackage[titletoc,page]{appendix}
\usepackage{bbm}
\usepackage{csquotes}

\usepackage{cmtt}
\usepackage[final]{showlabels}
\usepackage{graphicx}
\usepackage[maxbibnames=99,backend=biber,bibencoding=utf8,style=numeric,giveninits=true,url=false,doi=false,isbn=false]{biblatex}
\newtheorem{theorem}{Theorem}[section]

\newtheorem{lemma}[theorem]{Lemma}
\newtheorem{cor}[theorem]{Corollary}
\newtheorem{prop}[theorem]{Proposition}

\theoremstyle{definition}

\newtheorem{rem}[theorem]{Remark}

\renewcommand{\showlabelsetlabel}[1]{\raisebox{.5cm}{\kern 1.8cm\hbox{\color{blue}  \normalfont \small \llap{\textmtt{#1}}}}}

\newcommand{\R}{\mathbb{R}}
\newcommand{\C}{\mathbb{C}}

\renewcommand{\l}{\lambda}
\renewcommand{\d}{\delta}
\newcommand{\g}{\gamma}

\newcommand{\e}{\varepsilon}
\newcommand{\s}{\sigma}

\renewcommand{\i}{\textbf{i}}

\newcommand{\lv}{\lvert}
\newcommand{\rv}{\rvert}
\newcommand{\lvv}{\left\lvert}
\newcommand{\rvv}{\right\rvert}
\newcommand{\lb}{\left(}
\newcommand{\rb}{\right)}

\renewcommand{\O}{\mathcal{O}}
\newcommand{\supp}{\operatorname{supp}}
\newcommand{\dopp} [1] {\mathbbmss{#1}}

\renewcommand{\Im}{\operatorname{Im}}
\newcommand{\E}{\mathbb{E}}

\renewcommand{\k}{\kappa}

\newcommand{\N}{\mathbb{N}}

\newcommand{\mc}{\mathcal}

\newcommand{\mb}{\mathbf}
\newcommand{\Span}{\textup{span}}
\newcommand{\n}{{(n)}}

\newcommand{\BBr}{\mathbb{R}}

\newcommand{\dom}{\mbox{dom }}
\newcommand{\inte}{\mbox{Int}}

\begin{document}
    \title{Limit Theorems Under Several Linear Constraints}
    \author{Fabrice Gamboa}
    \address{Institut de Mathématiques de Toulouse, Universit\'e de Toulouse, ANITI Toulouse, and Invited Professor at Facultad de Ingenieria, Universidad de Medellin}
\email{fabrice.gamboa@math.univ-toulouse.fr}
\urladdr{https://www.math.univ-toulouse.fr/~gamboa/} 
    \author{Martin Venker}
    \address{School of Mathematical Sciences, Dublin City University}
\email{martin.venker@dcu.ie }
\urladdr{https://www.dcu.ie/maths/people/martin-venker} 
    \date{\today}
    \begin{abstract}
We study $n$ real-valued random variables subject to several linear constraints. Our main result is a weighted Central Limit Theorem, determining which linear combinations of these random variables are asymptotically normal as $n\to\infty$. Marginal distributions are also studied, showing that in the large $n$ limit random variables under linear constraints become i.i.d.~exponential under a rescaling. Our novel approach is based on a complex de Finetti theorem revealing an underlying independence structure, as well as on entropy arguments.
\end{abstract}
\maketitle  
\tableofcontents
\section{Introduction and quick tour on our results}
The Central Limit Theorem (CLT) is arguably one of the most important results in probability theory. Extending the classical assumptions of identical distributions and independence (i.i.d.) to non-identical distributions while keeping independence as an assumption is standard. Generalizing to dependent (real) random variables is less straightforward and more specific as many different types of dependence can be studied. This paper is concerned with studying the CLT and related results for real-valued random variables $X_1,\dots,X_n$ that are dependent through geometric constraints. As examples of such constraints, the constraint $\sum_{j=1}^n X_j^2\leq1(=1)$, corresponds to the unit ball (sphere), the constraint $\sum_{j=1}^n X_j=1$, $X_j\geq0$ corresponds to the standard simplex, and the $l_p$ balls (spheres) are built through the constraints $\sum_{j=1}^n\lv X_j\rv^p\leq1(=1)$. These examples are classical, yet have been widely studied in a variety of settings as we will outline below. Typically, these cases are studied using explicit representations in terms of independent random variables, e.g.~the uniform distribution on the unit sphere can be realized by setting $X_j:=Y_j/\|Y\|$, where $Y_1,\dots,Y_n$ are i.i.d.~standard normal and $\|\cdot\|$ denotes the Euclidean norm. The simplex case can be realized by normalizing i.i.d.~exponential random variables.

The aim of this paper is to develop a framework for dealing with random variables subjected to multiple constraints, focusing on several linear constraints
\begin{align}
	\sum_{j=1}^nA_{1j}X_j=b_1,\dots, \sum_{j=1}^nA_{mj}X_j=b_m, X_j\geq0,
\end{align}
for given constants $A_{ij}, b_i$, $1\leq i\leq m, 1\leq j\leq n$. It generalizes the simplex case $m=1$, $A_{1j}=1$, and offers an ideal testing ground to explore the effects of several simultaneous constraints. To mention a few of these effects that we will see in this paper, several constraints do not only induce independence but also lead to non-identical distributions of the $X_j$. More crucially, normalizing i.i.d.~random variables does not suffice in the context of several constraints, so a new approach is needed. And finally, the interplay between the different constraints is a completely new and interesting feature of this model. \\

Let us now formalize the setting. We will study vectors chosen at random from convex polytopes
$(K_n)_{n\geq1}$, where $K_n \subset \R^n$ is of the form
\begin{align}
	K_n:=\{x\in\R^n_{\geq0}:A^{(n)}x= b^{(n)}\}.\label{polytope}
\end{align}
Here, we set $\R^n_{\geq0}:=\{x\in\R^n:x\geq0\}$ to denote the non-negative orthant while for each $n$, $A^{(n)}$
is an $m\times n$ matrix with real entries and $b^{(n)}\in\R^m$. 
We assume, without loss of generality, that the situation is non-degenerate in the sense that $A^\n$ has full rank $m$ and that the relative interior of $K_n$ is non-empty.  In our work,  $m$ is considered as fixed while $n>m$ is growing. 
Furthermore, we assume that $K_n$ is compact. As we will see in Lemma \ref{lemma-positivity} below, the compactness is equivalent to us being able to choose $A^\n$ and $b^\n$ in \eqref{polytope} with all entries strictly positive. We equip $K_n$ with the uniform distribution $P_n$, which we define as the normalized $n-m$-dimensional Hausdorff measure on $K_n$. More precisely, letting $H^{n-m}$ denote the $n-m$-dimensional Hausdorff measure, we define for any Borel set $E\subset\R^n$,
\begin{align}
	P_n(E):=\frac{H^{n-m}(E\cap K_n)}{H^{n-m}(K_n)}.\label{def:Hausdorff}
\end{align}
Let now for $n>0$, $X^\n$ be a random vector 
with distribution $P_n$. We can view $X^\n$ as chosen at random, subject to the $m$ linear constraints
\begin{align}
	\sum_{j=1}^nA^\n_{ij}X^\n_j=b^\n_i,\quad i=1,\dots,m.\label{constraints}
\end{align}
Note that the components of $X^\n$ are not independent as they are coupled through the constraints $A^\n X^\n=b^\n$, and except for the case of a symmetric simplex, where $m=1$,  they are also not identically distributed. Finally, as $K_n$ is $n-m$-dimensional, the distribution of the vector $X^\n:=(X^\n_1,\dots,X^\n_n)$ is absolutely singular.

We are interested in the CLT by studying the weighted sum
\begin{align}
	S_n:=\sum_{j=1}^n \l_j^{(n)}(X^\n_j-\E X^\n_j),\label{def:S_n}
\end{align}
for asymptotic normality as $n\to\infty$, where $\E$ denotes expectation w.r.t.~$P_n$ and $\l_1^{(n)},\dots,\l_n^{(n)}\in\R$ are  general weights.
 Geometrically, for $\l^\n$ from the unit sphere, the sum $S_n$ is the scalar projection of the centered random vector $X^\n-\E X^\n$ onto $\l^\n$. One of the main questions addressed here is for which weights $\l_1^\n,\dots,\l_n^\n$, if any, the random variable $S_n$ is asymptotically Gaussian as $n\to\infty$. Notice that the classical choice of weights $\l_j:=n^{-1/2}$ for the central limit theorem in the i.i.d.~framework may not be relevant here, for two reasons. First, while $K_n$ is compact for each $n$, as $n\to\infty$ the polytope might grow, shrink, or fluctuate in size, and hence the magnitudes of each coordinate variable $X^\n_j$ can vary with $n$ and will in general not be uniform in $j$. Secondly, if the vector $\l^\n:=(\l_1^\n,\dots,\l_n^\n)$ is from the row space of $A^\n$, then $S_n$ of \eqref{def:S_n} is deterministic and thus no (non-degenerate) CLT can hold.\\
 
 Our main result is Theorem \ref{thrm1} below where we show, under suitable assumptions, a general central limit theorem for $S_n$. Here, the asymptotic standard deviation is obtained as the norm of the orthogonal projection of the standardized weight vector onto a space defined by the constraints. We also study the marginal densities of finitely many $X_j^\n$ in Theorem \ref{thrm_marginal} and show that they become, appropriately rescaled, independent exponential in the limit $n\to\infty$. Our results are based on a novel approach to distributions under constraints. This approach reveals an underlying independence structure of the dependent $X_j^\n$ that should be of independent interest and that we present as a complex de Finetti type result in Theorem \ref{thrm_mixture}.  We will also see that a key ingredient to understanding the uniform distribution on $K_n$ is based on a maximum entropy principle to determine the probabilistic barycenter of the polytope (see Proposition \ref{Proposition:w} and Corollary \ref{Cor_Entrop}).

\section{Main results}
\subsection{Notation}
To alleviate notation, we will omit the superscript $(n)$ in all objects we are dealing with (e.g. $A^\n$ will be written as $A$ and $X^\n$ as $X$).

To answer the question of good weights leading to a CLT for $S_n$ of \eqref{def:S_n}, consider the set of probability measures on $\R^n_{\geq0}$ whose mean vector exists and lies in the polytope $K_n$. By general theory it is known that there is a unique probability measure $P_{n,w}$ in this set that maximizes entropy, given by
\begin{align}
	dP_{n,w}(x)=\lb\prod_{j=1}^{n}w_j\rb e^{-\sum_{j=1}^{n}  w_jx_j}\dopp1_{\R^{n}_{\geq0}}(x)dx,\label{def:P_n,w}
\end{align}
for some parameters $w_1,\dots,w_n>0$ (see Section \ref{su_maxent}). We will establish that asymptotically $S_n$ of \eqref{def:S_n} will behave similarly under $P_n$ and under $P_{n,w}$. Note that for a random vector $X$ with distribution $P_{n,w}$, the random variables $X_j$ are independent and have exponential distributions with parameters $w_j$, while under $P_n$ the $X_j$ are dependent and not exponentially distributed. In Section \ref{su_maxent}, we will prove interesting properties of $w:=(w_1,\dots,w_n)$, in particular that it lies in the row span of $A$, i.e.~$w=u^tA$ for some $u\in\R^m$, the superscript $t$ denoting transposition. This implies that the density of $P_{n,w}$ is constant on $K_n$, in line with being a good absolutely continuous approximation to the uniform distribution $P_n$ on $K_n$, which is itself absolutely singular. We note in passing that while the $w_j$ do not have an explicit form in terms of $A$ and $b$, it is easy to provide a numerical approximation, a very simple Python program is provided in the appendix for the reader's convenience.

It will prove convenient to rescale the coordinates in units of $w_j$. To this end, we define the matrix $\tilde A$ as 
\begin{align}\label{def:Atilde}
	\tilde A_{ij}:=\frac{A_{ij}}{w_j},\quad 1\leq i\leq m,\ 1\leq j\leq n.
\end{align}
The mean of $P_{n,w}$ being the vector $(\frac{1}{w_j})_{j=1,\cdots,n}$, this change of scale means that we normalize the polytope so that its new entropy center is the vector whose coordinates are all equal to $1$.  
To state our first main theorem we will further have to fix a canonical representative for $(A,b)$ in the characterization \eqref{polytope} of $K_n$. Indeed, if $(A,b)$ is a representation of $K_n$, then so is $(MA,Mb)$ for any invertible $m\times m$ matrix $M$. We define
\begin{align}\label{def:Ahat}
	\hat A:=(\tilde A \tilde A^t)^{-\frac12}\tilde A,\quad \hat b:=(\tilde A \tilde A^t)^{-\frac12}b.
\end{align} 
To briefly motivate this choice we remark that for $S_n$ of \eqref{def:S_n} to approach a normal distribution it is important to measure the {\it gap} of the linear statistic $\sum_{j=1}^n \l_jX_j$ to the $m$ constraints $\sum_{j=1}^nA_{ij}X_j=b_i$, $i=1,\dots,m$. Indeed, if the linear combination $\sum_{j=1}^n \l_jX_j$ is too close to the span of of the equations defining the constraints, then some or all of the randomness of $S_n$ is annihilated, preventing a CLT to hold. This problem will lead to an assumption on the weights $\l_j$. 

A bigger challenge is however to detect whether \textit{asymptotic degeneration} is occurring in the constraints.  By asymptotic degeneration we mean that there are finitely many variables that, as $n\to\infty$, play a much more prominent role in the fulfillment of one or more constraints, compared to the other variables.
This could effectively reduce the dimension of the polytope. The matrix $\hat A$ is much better suited to detect such asymptotic degeneration than $A$ or $\tilde A$ as under $\hat A$ the $m$ rescaled constraints are orthonormal (since $\hat A\hat A^t=I$). This can be understood as a standardization of $A$, decorrelating the constraints and unifying the magnitude of their representing equations. $\hat A$ also encapsulates the $b$-dependence of the polytope and thus allows us to make one simple assumption on $\hat A$ ruling out asymptotic degeneration: we will assume that $\|\hat A\|_{\max}=o(1)$ as $n\to\infty$, where $\|\cdot\|_{\max}$ denotes the maximum norm. A detailed discussion of this assumption is provided in Section \ref{sec:example}. For now it suffices to mention that $\|\hat A\|_{\max}=o(1)$ holds \textit{generically} in the sense that if $A$ and $b$ are chosen at random in a natural way, then $\|\hat A\|_{\max}=o(1)$ with probability going to 1 as $n\to\infty$. We will also see that a simple way to fulfill this assumption is to choose $A$ such that each column in $A$ is repeated infinitely many times as $n\to\infty$. A very simple example of this type would be 
	\begin{align*}
	A:=\begin{pmatrix}
		1&2&1&2&\dots&2\\
		3&4&3&4&\dots&4
	\end{pmatrix},\quad b:=\begin{pmatrix}2\\5\end{pmatrix}.
\end{align*} 
 
\subsection{Statements of main results}
Recall the definition of $S_n$ in \eqref{def:S_n}. Let us normalize the weights  by setting
\begin{align}
\hat\l_j:=\frac{\l_j}{w_j},\quad j=1,\dots,n.
\end{align}
As mentioned above, we are interested in determining which weight vectors $\l$ lead to a CLT for $S_n=\sum_{j=1}^n \l_j(X_j-\E X_j)$. Let us recall that for any $\l$ in the row span of $A$, $S_n$ is not random as in that case $S_n$ is essentially a combination of the constraints \eqref{constraints}. The following theorem is our main result which shows asymptotic normality of $S_n$ as $n\to\infty$ as long as $\l$ is sufficiently bounded away from the row span of $A$. This is quantified nicely in terms of $\hat A$ rather than $A$, further emphasizing the role of the rescaled and decorrelated version $\hat A$ of $A$.  Recall that we use $\|\cdot\|$ for the Euclidean norm.
\begin{theorem}\label{thrm1}
	Assume that $\|\hat A\|_{\max}=o(1)$ holds as $n\to\infty$. Assume further that
	\begin{align}
		 \|\hat\l\|_{\max}= o(1) \text{ as } n\to\infty, \ \limsup_{n\to\infty}\|\hat\l\|<\infty, \text{ and  }\liminf_{n\to\infty}\|P_{\operatorname{ker(\hat A)}}\hat\l\|^2>0,\label{conditions_maintheorem}
	\end{align} 
where $P_{\operatorname{ker(\hat A)}}$ denotes the orthogonal projection to the kernel of $\hat A$.
 Then, with $\s:=\|P_{\operatorname{ker(\hat A)}}\hat\l\|$, $\s^{-1}S_n$  converges in distribution to a standard normal random variable as $n\to\infty$. In \eqref{def:S_n}, the centering $\E X_j^\n$ w.r.t.~$P_n$ may be replaced by the centering $\E_{P_{n,w}}X_j^\n$ w.r.t.~$P_{n,w}$ without altering the result.	\\
	The variance $\|P_{\operatorname{ker(\hat A)}}\hat\l\|^2=\|P_{\operatorname{ker(\tilde A)}}\hat\l\|^2$ can also be written as
\begin{align}
	\|P_{\operatorname{ker(\hat A)}}\hat\l\|^2=\|\hat\l\|^2-\|\hat A\hat\l\|^2.
\end{align}
\end{theorem}
\begin{rem}
The assumption $\|\hat\l\|_{\max}= o(1)$ above is in general necessary for the asymptotic normality as $\l_jX_j=\hat\l_jw_jX_j$, and $w_jX_j$ is close to an exponentially distributed random variable with parameter 1 (see Theorem \ref{thrm_marginal} below). It guarantees that no single variable $X_j$ has a non-negligible influence on the sum $S_n$. The two remaining assumptions in \eqref{conditions_maintheorem} ensure that the variance $\s^2$, which  depends on $n$, is bounded above and bounded away from 0. 
		 In particular, $\liminf_{n\to\infty}\|P_{\operatorname{ker(\hat A)}}\hat\l\|^2>0$ in \eqref{conditions_maintheorem} ensures that $\sum_{j=1}^n \l_jX_j$ is sufficiently far away from the constraints. Indeed, any $\hat\l$ can be written uniquely as the sum $\hat\l_1+\hat\l_2$ of an element $\hat\l_1$ of the row space of $\hat A$ and an element $\hat\l_2$ of the kernel of $\hat A$. However, as the rows of $\hat A$ describe the constraints, the sum $\sum_{j=1}^n \hat\l_{1,j}(X_j-\E X_j)$ is deterministic, thus explaining the occurrence of $P_{\operatorname{ker(\hat A)}}$.
\end{rem}

Let us now have a look at the asymptotic marginal distributions of the $X_j$. While for finite $n$ the $X_j$ are dependent and clearly not exponentially distributed (they are bounded), in the limit $n\to\infty$ they become independent exponential random variables under a rescaling.

\begin{theorem}\label{thrm_marginal}
	Assume that $\|\hat A\|_{\max}=o(1)$ holds as $n\to\infty$. Then, for any positive integer $k$ and any $j_1,\dots,j_k\in\{1,\dots,n\}$ as $n\to\infty$
	\begin{align}
		\lb w_{j_1}X_{j_1},\dots,w_{j_k}X_{j_k}\rb\to(Y_1,\dots,Y_k)
	\end{align}
in distribution, where $Y_1,\dots,Y_k$ are i.i.d.~random variables with exponential distribution with parameter 1.
\end{theorem}
Our results rely on a novel method of expressing useful quantities of the uniform distribution $P_n$ as a mixture of similar quantities of independent distributions. The essence of this approach is shown by the following result which may be seen as a \textit{complex de Finetti theorem}. 

 We use $\i$ for denoting the imaginary unit and for a matrix $A$ we denote by $A_{l\bullet}$ its $l$'th row and by $A_{\bullet j}$ its $j$'th column.
\begin{theorem}\label{thrm_mixture}
    Let $A$ be such that its rank can not be reduced by removing one column. Then, for any set $E$ of the form $E=\bigtimes_{j=1}^n E_j$ with the $E_j$ being intervals in $\R$, 
    \begin{align}\label{equ:mixture}
        P_n(E)=\int_{\R^m} P_{n,w,\eta}(E)dP_{n,\textup{mix}}(\eta),
    \end{align}
    where $P_{n,w,\eta}$ is the complex-valued measure on $\R^n$ defined by
    \begin{align}\label{def:complex_measure}
        & P_{n,w,\eta}(B):=\frac{1}{Z_{n,w,\eta}}\int_Be^{-\sum_{j=1}^{n}  (w_j-\i \langle\eta, A_{\bullet j}\rangle)x_j}\dopp1_{\R^{n}_{\geq0}}(x)dx,\quad B\in\mc B_{\R^n},
    \end{align}
    $Z_{n,w,\eta}$ is the normalizing constant 
    \begin{align}
        &Z_{n,w,\eta}:=\int_{\R^n}e^{-\sum_{j=1}^{n}  (w_j-\i \langle\eta, A_{\bullet j}\rangle)x_j}\dopp1_{\R^{n}_{\geq0}}(x)dx,
    \end{align}
    and the complex-valued mixture measure $P_{n,\textup{mix}}$ on $\R^m$ is defined as
    \begin{align}
        dP_{n,\textup{mix}}(\eta)=\frac{1}{Z_{n,\textup{mix}}}\E_{P_{n,w}}e^{\i \sum_{j=1}^n\langle \eta,A_{\bullet j}\rangle X_j-\i\langle\eta,b\rangle}d\eta,\quad   Z_{n,\textup{mix}}:=\int_{\R^m}\E_{P_{n,w}}e^{\i \sum_{j=1}^n\langle \eta,A_{\bullet j}\rangle X_j-\i\langle\eta,b\rangle}d\eta.
    \end{align}
    In particular, $0<\lv Z_{n,\textup{mix}}\rv<\infty$ and $0<\lv Z_{n,w,\eta}\rv<\infty$ for any $\eta\in\R^m$.
\end{theorem}
\begin{rem}\label{rem_de_Finetti}
    \begin{enumerate}
        \item Theorem \ref{thrm_mixture}  resembles the classical de Finetti theorem expressing infinite exchangeable sequences as mixtures of i.i.d.~sequences. Note however that a classical de Finetti representation does not generally hold for finite (exchangeable) sequences \cite{DiaconisFreedman}. Moreover, except for the case $m=1$ of one constraint with equal entries $A=(c,c,\dots,c),\ c>0$, which corresponds to a symmetric simplex, the random vector $X$ chosen uniformly from $K_n$ does not consist of finitely exchangeable random variables. Recently the notion of weighted exchangeability has been introduced in the infinite setting in \cite{BCRT24}, and been studied in the finite setting in \cite{Tang2023}. Briefly speaking, weighted exchangeable random variables are classically exchangeable after a change of variables. We believe this notion to cover the case $m=1$ of one constraint with not necessarily equal weights, but not the case $m>1$, as several linearly independent constraints lead to random variables $X_j$ with a different dependence structure.
While the measures $P_{n,w,\eta}$ in \eqref{equ:mixture} are not probability measures, they share the main algebraic advantage of factorization into marginal measures with probability distributions of independent random variables. 

One drawback of having complex-valued measures in the representation \eqref{equ:mixture} is however that it does not apply to all Borel sets $E$. One natural restriction on $E$ comes from the absolute singularity of $P_n$ and the absolute continuity of $P_{n,w,\eta}$, necessitating that  $E$ is a continuity set for $P_n$, i.e.~$P_n(\partial E)=0$, where $\partial E$ denotes the boundary of $E$. The main restriction on $E$ however stems from the lack of uniform boundedness of complex-valued measures, even if normalized to total mass 1. More precisely, for a general Borel set $E$, the quantity $\lv P_{n,w,\eta}(E)\rv$ is not bounded in $\eta$. The reason for this is that for a rough set $B=E$ the decay in $\eta$ in the Fourier transform in \eqref{def:complex_measure} can be much slower than the decay of its normalizing constant $Z_{n,w,\eta}$. The range of sets $E$ for which \eqref{equ:mixture} holds could be extended to continuity sets $E$ with smooth boundary using results from harmonic analysis. As we will not apply Theorem \ref{thrm_mixture} directly but in the form of Corollary \ref{cor_Bartlett} below for the characteristic function of $S_n$, we will not pursue this generalization here.
\item Our proof below, see Lemma \ref{lemma_integrability}, will show that the rank condition on $A$ is in fact not just sufficient but also necessary for the representation \eqref{equ:mixture} to hold. It guarantees just enough decay in $\eta$ to ensure the finiteness of the $\eta$-integral in \eqref{equ:mixture}. The rank condition on $A$ is equivalent to the simple geometric condition that no variable $X_j$ is constant for $X$ uniform on $K_n$, see Lemma \ref{lemma_integrability} below. 

 Making stronger assumptions on $A$ like requiring each column to be repeated once or several times would be another way to enlarge the class of $E$ for which \eqref{equ:mixture} holds.
    \end{enumerate}
\end{rem}
\subsection{Related results}
 In recent years there has been a lot of interest in the limiting behavior of sums of random variables derived from high-dimensional geometric objects. This partly goes back to Klartag's celebrated result \cite{Klartag} stating that if $X=(X_1,\dots,X_n)$ is a vector taken at random from a convex body in $\R^n$ such that it is centered and has the identity as covariance matrix, projections $\langle \l,X\rangle=\sum_{j=1}^n\l_jX_j$  are, for $n$ large, close to a standard normal distribution. This impressive result is a CLT for uncorrelated centered random variables and holds for most directions $\l\in\R^n,\,\|\l\|=1$, in the sense that if a $\l$ with $\|\l\|=1$ is picked at random, with high probability the projection of $X$ will tend to a normal distribution as $n\to\infty$. 
 Klartag's result \cite{Klartag} is one in a line of general results, we mention here \cite{Sudakov,vonWeizsacker,Bobkov,Meckes} and references therein. Our setup does not fall in the scope of these works, for example due to our $X$ being supported in a lower-dimensional space. Generally speaking, our setup is more specialized but as a consequence our results are far more explicit.

Random variables under a single specific constraint have been studied for a long time. For a historical perspective on the genesis of the very first result on the high-dimensional $l_2$-sphere, corresponding to the constraint $\sum_{j=1}^n X_j^2$, we refer to Section 6 of \cite{DiaFr}. Attributed to Borel and Poincar\'e, the result states that asymptotically any finite dimensional projection of a uniform vector in the high dimensional $l_2$-sphere is Gaussian. A classical proof can be based on the representation of a uniformly distributed random vector on the $l_2$-sphere as a normalized standard Gaussian vector. Generalizing this probabilistic representation to the $l_p$-ball led to interesting geometric results in \cite{BartheGuedon}. 
In recent years, there has been renewed interest in studying the asymptotic properties related to the Borel-Poincaré theorem for the $l_p$-ball in the case of random projections. This was driven by the innovative work \cite{GKR} where large deviations for the quenched and annealed regimes are provided. This last paper was followed by a series of papers by Kabluchko, Prochno, Th\"ale \textit{et al.}~which extended and generalized the results in many directions ranging, for example, from Schatten classes \cite{kabluchko2020sanov} to the Lorentz ball \cite{kabluchko2024probabilistic}. The simplex as a section of the $l_1$-ball is discussed in \cite{Bacietal}. We refer to \cite{prochno2019geometry} for a summary and context of some of these results. The case of Stiefel manifolds is studied in \cite{kim2023large}.  Often, in these works, the key ingredient is the use of the representation of the uniform distribution on the ball via certain gamma or Dirichlet distributions. This approach also led to asymptotic results for a family of distributions on the $l_p$-ball in \cite{Bartheetal}. In the latter work, a link is made with the case where the high-dimensional convex set studied probabilistically is a moment space.  Borel-Poincaré theorems (and their refinements) in these so-called random moment problems have been widely studied in \cite{GamboaLozada}, \cite{Lozada}, \cite{dette2012distributions}, \cite{gamboa2012large}, \cite{DTV1,DTV2}, initiated by the pioneering work of Chang {\it et al} (see \cite{CKS}). In all these works, the main ingredient for the asymptotic study is the existence of a coordinate system based on the Knothe-Rosenblatt  map \cite{knothe1957contributions},  called canonical moments (see \cite{dette1997theory}), for which the push-forward probability measure of the uniform law is a product measure. In addition, the components of this product can be constructed from gamma or Dirichlet distributions. 

The present work adds to all of these results insofar that it develops an approach that does not rely on explicit highly-specific representations in terms of independent random variables, and crucially opens up the analysis of random variables under several constraints. As we will see in the next section, the interplay between the different constraints can lead to interesting phenomena that cannot be observed in the single constraint case.\\

The rest of the paper is organized as follows. In the next section we provide a detailed discussion of the assumption $\|\hat A\|_{\max}=o(1)$ that is used for our asymptotic results. In particular, Proposition \ref{prop:examplebis} gives a very simple way of constructing examples of $K_n$ fulfilling this assumption, but also shows that $\|\hat A\|_{\max}=o(1)$ is a generic property. Section \ref{su_Unif_Bart} starts the analysis of $P_n$ by providing an absolutely continuous approximation to the absolutely singular $P_n$. This approximation is then utilized in Section \ref{sec:deFinetti} to prove the complex de Finetti representation Theorem \ref{thrm_mixture}. As a corollary we also present a very useful Bartlett type representation of the characteristic function of $S_n$ that we will use later on to prove our asymptotic results. Section \ref{su_maxent} contains the entropy analysis of $K_n$, in particular existence and properties of the parameters $w_j$ are established. These properties are then used in Section \ref{sec:proofexample} to provide a proof of Proposition \ref{prop:examplebis}. The proofs of Theorems \ref{thrm1} and \ref{thrm_marginal} are given in Section \ref{su_AsA}. A simple Python code for a numerical computation of the parameters $w_j$ is provided in the appendix.

\section{On the assumption $\|\hat A\|_{\max}=o(1)$}\label{sec:example}
The assumption $\|\hat A\|_{\max}=o(1)$ serves as an asymptotic analog of the condition that the rank of $A$ cannot be reduced by removing one column, which is a necessary and sufficient condition for the non-asymptotic complex de Finetti representation \eqref{equ:mixture}. As discussed in the preceding remark, the rank condition on $A$ guarantees the finiteness of the integral in \eqref{equ:mixture} for finite $n$. For the large $n$ asymptotics however, this condition is not sufficient as the scaling by $w_j$, or rather via $\hat A$, will in general dampen the decay of the characteristic function. This is in analogy of a classical proof of the classical CLT for i.i.d.~standardized random variables, where one shows that $\E e^{i\frac t{\sqrt n}\sum_{j=1}^n X_j}=\phi(t/\sqrt n)^n$ converges to the characteristic function of a standard normal distribution, $\phi$ being the characteristic function of the standardized random variable $X_1$. Any decay of $\phi$ is dampened by the $n^{-1/2}$ scaling and only raising it to the power of $n$ yields the desired limit. \\
In our case the same problem occurs but it is complicated by the scaling being different for each $j$, and the need to integrate over the parameter $\eta$, see again \eqref{equ:mixture}. To see that an analog of the rank condition of Theorem \ref{thrm_mixture} is not enough, consider the example $b:=(2,1)^t$,
	\begin{align}\label{example1}
	A:=\begin{pmatrix}
		1&1&1&1&\dots&1\\
		1&2&\frac{a_1}n&\frac{a_2}n&\dots&\frac{a_{n-2}}n
	\end{pmatrix},
\end{align} 
where the $0<a_j<1$ are all different. Here, any $n-2$ columns can be removed without lowering the rank of $A$, yet asymptotically the first two columns are indispensable. We call a situation with such a dependence on a finite number of variables asymptotically degenerate. Note that the asymptotic degeneration of \eqref{example1} depends on $b$. If for $A$ of \eqref{example1} we have $b:=(1,n)^t$ instead, there is no dependence on a finite number of variables as $n\to\infty$. This shows that a good condition on $A$ should take $b$ and different magnitudes of the entries of $A$ into account. This leads the way to making an assumption on $\hat A$ instead of $A$.

Let us now discuss the assumption  $\|\hat A\|_{\max}=o(1)$ more closely.
We first remark that thanks to $\hat A\hat A^t=I$, for any $c>0$ the number of entries of $\hat A$ larger than $c$ is bounded in $n$. We will see in a few examples below that non-vanishing entries occur in cases where finitely many variables play a prominent role in the sense that they are non-negligible, retaining more dependence than the bulk of the variables. Our assumption $\|\hat A\|_{\max}=o(1)$ rules out such a case, ensuring that each single variable is negligible. We will more precisely show that it implies the following property:

\vspace{1em}

\noindent
\textbf{Property A:}
For each $K\in\N$ there exist pairwise disjoint nonempty subsets $I_1,\dots,I_{K}\subset\{1,\dots,n\}$ and an $n$-independent constant $C>0$ such that for all $l=1,\dots,K$ and all $n$ we have
\begin{align}
	\det\lb\hat A_{I_l}(\hat A_{I_l})^t\rb\geq C,\label{det-condition}
\end{align}
where $\hat A_{I_l}$ is the matrix whose columns are the columns $\hat A_{\bullet j}$, $j\in I_l$, ordered increasingly in $j$ for definiteness.

\vspace{1em}
Property A may be seen as an analog of the rank condition on $A$ in Theorem \ref{thrm_mixture}. Briefly speaking it allows for the repeated removal of a significant part of the variables while the remaining variables are still significant, the term ``significant'' here referring to the determinant property \eqref{det-condition}.

The bound	\eqref{det-condition} can also be understood as a volume condition. Following \cite{MikhalevOseledets18}, $\sqrt{\det\lb\hat A_{I_l}(\hat A_{I_l})^t\rb}$ is the volume of the rectangular matrix $\hat A_{I_l}$. In geometric terms, \eqref{det-condition} means that the ratio of the volume of the $m$-dimensional image of the $\lv I_l\rv$-dimensional unit ball under $\hat A_{I_l}$ to the volume of the unit ball in $\R^{\lv I_l\rv}$ is bounded away from 0.
	
	From a technical point of view, Property A plays a crucial role in our analysis as \eqref{det-condition} ensures uniform integrability (as $n\to\infty$) of an expression closely related to
	\begin{align}
	\lvv\E_{P_{n,\mb1}}e^{\i \sum_{j=1}^n\langle \eta,\hat A_{\bullet j}\rangle X_j-\i\langle\eta,\hat b\rangle}\rvv
	\end{align}
	over $\eta\in\R^m$, where $\hat A_{\bullet j}$ denotes the $j$-th column of $\hat A$ and $\mb1$ is the vector of $w_j:=1$ for $j=1,\dots,n$. This will be made precise starting from  \eqref{bound_characteristic_function} below. This uniform integrability is necessary to utilize the product structure revealed in Theorem \ref{thrm_mixture}.  Let us remark that the Brascamp-Lieb inequality, which provides strong estimates for integrals of such functions, is not strong enough in this case to achieve uniform integrability, even when involving the optimal Brascamp-Lieb constant \cite{CarlenLiebLoss}. This indicates the problem to be rather involved, relying on delicate connections between the columns of $\hat A$. Interestingly, for the proof of the CLT, Theorem \ref{thrm1}, the mentioned uniform integrability only requires Property A to hold for $K=m+1$. Replacing $\|\hat A\|_{\max}=o(1)$ by this weaker assumption would allow for non-negligible variables that would not become independent or exponentially distributed in the limit $n\to\infty$. We plan to look at this more general case in a future work.\\

The following proposition links our assumption $\|\hat A\|_{\max}=o(1)$ to Property A and also presents two simple situations for which $\|\hat A\|_{\max}=o(1)$ holds. 
Indeed, expectedly, under enough exchangeability of columns the assumption is satisfied. It also shows that $\|\hat A\|_{\max}=o(1)$ is a \textit{generic} property, i.e.~if for each $n$ we choose $A$ and $b$  at random, then our assumption holds with probability approaching 1 as $n\to\infty$.  To prepare this statement,   
   let $P$ be a probability measure on $\BBr_{\geq 0}^m$. For the sake of simplicity, we will assume that $U:=\supp P$ (the support of $P$) is a compact subset of $\BBr_{\geq 0}^m$. We assume further that the closed convex hull of  $U$ denoted by $\mathcal{U}$ does not contain $0$ and is not a subset of any hyperplane of $\BBr^m$.  
  Recall that for a matrix $A$,  $A_{l\bullet}$ denotes its $l$'th row and $A_{\bullet j}$ its $j$'th column. Let the columns $A_{\bullet1},A_{\bullet2},\dots,A_{\bullet n}$ be i.i.d. with distribution $P$ and consider the random $m\times n$ matrix
   $A:=(A_{\bullet1}A_{\bullet2}\cdots A_{\bullet n})$. 
   
   Having chosen $A$, we now turn to $b$. As we demonstrated in the example \eqref{example1}, the choice of $b$ impacts on the question of asymptotic degeneracy of the constraints, hence some care is needed. To this end, for a $v\in\BBr^m$ such that 
   the linear form $\langle v,\cdot\rangle$ is positive on $U$, we define 
   $G(v):=-\int_U\log\langle v,u\rangle P(du)$. As $G$ is a convex function we may extend it semi-continuously to vectors $v$ such that the linear form $\langle v,\cdot\rangle$ is non-negative on $U$ (see \cite{rockafellar2015convex}, p.~52). We finally set
   $$ \dom G:=\{v\in\BBr^m:\;\langle v,\cdot\rangle\geq 0\mbox{ on $U$ and } G(v)<+\infty\}.$$
   Lastly, for $v\not\in\dom G$, we set $G(v):=+\infty$. 
   Notice that $G$ is infinitely differentiable on
   $$\inte(\dom G)=\{v\in\BBr^m: \langle v,\cdot\rangle >0 \mbox{ on $U$}\}.$$ 
   This last set is not void as it contains any vector with  positive components.
   Moreover, to compute a derivative at any order for these points, we may permute integration and derivation. Hence, for example for 
   $v\in\inte(\dom G)$, we have
   $$\nabla G(v)=\int_U\frac{uP(du)}{\langle v,u\rangle}.$$
   As we will see, when considering a random polytope built with the random matrix $A$, the {\it good} set of admissible $b$ is 
   $$\mathcal{B}:= \{\nabla G(v):\; v\in\BBr^m\mbox{ with }\;\langle v,\cdot\rangle> 0\mbox{ on $U$}\}=
   \{\nabla G(v):\; v\in\inte(\dom G)\}.$$ 
   Notice that, in the general case, the map 
   $\nabla G:\inte(\dom G)\rightarrow \mathcal{B}$ 
   is one to one (Theorem 26.4 p. 256 in \cite{rockafellar2015convex}). Hence, for $b\in \mathcal{B}$ we may define $v(b)$ via the equation $b=\nabla G(v(b))$.

\begin{prop}
\label{prop:examplebis}\noindent
\begin{enumerate}
		\item If $\|\hat A\|_{\max}=o(1)$ as $n\to\infty$, then Property A holds.
		\item If every column of $A$ appears in $A$ at least $K$ times, then Property A holds for that $K$. If for every column of $A$ the number of repetitions of that column goes to infinity as $n\to\infty$, then $\|\hat A\|_{\max}=o(1)$.  
		\item If $(A,b)$ are chosen at random as described above, i.e.~the columns $A_{\bullet j}$, $j=1,\dots,n$, are i.i.d.~random vectors in $\R^m$, with common distribution $P$ supported on a compact subset $U$ of $\R^m_{\geq0}$ such that the convex hull of $U$ does not contain $0$ and is not contained in a hyperplane, and the vector $b$ is from the set $\mc B$, then 
with probability converging to 1 
as $n\to\infty$ 
the equation $AX=b,\ X\geq0$, defines a compact polytope and $\|\hat A\|_{\max}=o(1)$ holds.
\end{enumerate}
\end{prop}
\begin{rem}
For the sake of simplicity and conciseness, we opted for showing part (3) of Proposition \ref{prop:examplebis} with convergence in probability.
Nevertheless, the same statement with almost sure convergence could  been obtained using convexity arguments as the ones developed in Lemma 1 of \cite{Daga} and Theorem 2.1 and 2.2 of  \cite{Gaga}, but this would have required much more space.
\end{rem}

As mentioned before, from a non-technical view the assumption $\|\hat A\|_{\max}=o(1)$ ensures the absence of asymptotic degeneration.	To appreciate the intricacies of asymptotic degeneration, let us have a look at an example.
	
\begin{rem}
Property A for a fixed $K$, e.g.~for $K:=m+1$, implies that each constraint must involve at least $K+1$ variables in a non-negligible way. To illustrate this, take as an example $m=2$, $b=(2,1)^t$ and the matrix
	\begin{align}
		A:=\begin{pmatrix}
			1&1&1&1&\dots&1\\
			1&1&1&0&\dots&0
		\end{pmatrix},
	\end{align} 
	leading to 
	\begin{align}
		\hat A=\begin{pmatrix}
			0&0&0&\frac{1}{\sqrt{n-3}}&\dots&\frac{1}{\sqrt{n-3}}\\
			\frac{1}{\sqrt3}&\frac{1}{\sqrt3}&\frac{1}{\sqrt3}&0&\dots&0
		\end{pmatrix},\quad \hat b=\begin{pmatrix}
			\sqrt n\\
			\sqrt n
		\end{pmatrix}.
	\end{align}
	In the computation of $\hat A$ we used that $w$ is the only vector in the row span of $A$ such that $A\frac1w=b$ holds, where we define $(\frac1w)_j:=\frac{1}{w_j}$, $j=1,\dots,n$. We will prove this fact in Proposition \ref{Proposition:w}.
	We can see from $\hat A$ that the first three variables remain significant as $n\to\infty$, while each of the other variables becomes asymptotically negligible. 
	In this example, column subsets $I_1,I_2,I_3$ can be chosen satisfying \eqref{det-condition} by each of the three subsets containing one of the first three columns and sufficiently many of the remaining columns. 
	
	In the above example the form of $\hat A$ resembles greatly the one of $A$, or rather of the equivalent representation $(A',b')$,
	\begin{align}
		A':=\begin{pmatrix}
			0&0&0&1&\dots&1\\
			1&1&1&0&\dots&0
		\end{pmatrix},\quad b':=(1,1)^t.
	\end{align} 
	In particular, it suggests that dependence of a constraint on a finite number $k$ of variables translates into $k$ asymptotically non-vanishing variables in $\hat A$.
	However, this resemblance is not present in general: consider the slight variation
	\begin{align}
		A:=\begin{pmatrix}
			1&1&1&1&\dots&1\\
			1&2&3&0&\dots&0
		\end{pmatrix},\quad b:=(2,1)^t.
	\end{align} 
	As we will see in Proposition \ref{Proposition:w}, we have $w_j:=\l_1A_{1j}+\l_2A_{2j}$ for any $j$, where $\l_1,\l_2\in\R$. Using this, it is straightforward to see that $\l_1$ must be of order $n$ since $\sum_{j=1}^n A_{1j}\frac{1}{w_j}=2$. Now, as also $\sum_{j=1}^n A_{2j}\frac{1}{w_j}=A_{21}\frac{1}{w_1}+A_{22}\frac{1}{w_2}+A_{23}\frac{1}{w_3}=1$, we must have $\l_2$ of order $-n$ in such a way that at least one $w_j$, $j=1,2,3$, is of order 1. The only possibility for this is (since $w_j>0$) that $\l_2\sim -\frac{\l_1}3$. Overall, we find that 
		\begin{align}
		\hat A=\begin{pmatrix}
			\Theta(\frac{1}{\sqrt n})&\Theta(\frac{1}{\sqrt n})&\Theta(1)&\Theta(\frac{1}{\sqrt n})&\dots&\Theta(\frac{1}{\sqrt n})\\
			\Theta(\frac{1}{\sqrt n})&\Theta(\frac{1}{\sqrt n})&\Theta(1)&\Theta(\frac{1}{\sqrt n})&\dots&\Theta(\frac{1}{\sqrt n})
		\end{pmatrix},
	\end{align}
	where we write $f=\Theta(g)$ if $f=\O(g)$ and $g=\O(f)$. Here, notably, in $\hat A$ only the third variable is of order 1, while the second constraint relies on three variables. It shows that our assumption $\|\hat A\|_{\max}=o(1)$ cannot easily be replaced by an assumption on the number of asymptotically non-vanishing variables in $\hat A$.

\end{rem}

\section{An approximation of the uniform distribution}
\label{su_Unif_Bart}
The aim of this section is the construction of a useful representation of the uniform distribution $P_n$ on $K_n$. 
Recall from \eqref{def:Hausdorff} that the uniform distribution on $K_n$ is defined as the normalized $n-m$-dimensional Hausdorff measure on $K_n$.
  This definition has the disadvantage of adding no analytic or probabilistic structure to the problem. Instead, we will provide a representation that encodes the constraints in the parameters of independent exponentially distributed random variables, and allows us to work with independent random variables in the full space $\R^n$. We start by a simple lemma showing that we can assume the entries of $(A,b)$  in \eqref{polytope} to be positive.
\begin{lemma}\label{lemma-positivity}
	Any compact convex polytope $K_n\subset\R^{n}_{\geq0}$ with $K_n\cap\textup{Int}(\R^n_{\geq0})\not=\emptyset$ of the form $K_n:=\{x\in\R_{\geq0}^{n}:Ax= b\}$ admits a representation $( \bar A,\bar b)$ with $\bar A_{lj}>0$, $l=1,\dots,m,j=1,\dots,n$ and $\bar b_l>0$, $l=1,\dots,m$.
\end{lemma} 
\begin{proof}
	This is a simple consequence of the strong duality theorem in linear programming. Since $K_n$ is compact and non-empty, the linear program $Ax=b,x\geq0, \max \mathbf 1x$, has an optimal solution for $\mathbf 1:=(1,\dots,1)\in\R^n$. Hence, by strong duality, also the dual linear program $y^tA\geq \mathbf 1, \min y^tb$, has an optimal solution. In particular, there is a $y\in\R^m$ with $y^tA_{l\bullet}\geq 1$, $l=1,\dots m$. As $(y^tA)_j>0$ for each $j$ and there are solutions $x$ of $Ax=b$ with $x_j>0$ for all $j$, we clearly have $y^tb>0$. We may now replace some row $A_{l\bullet}$ by $y^tA$ and $b_l$ by $y^tb$. Now adding multiples of this new  row   to the others to make them positive (while doing the same with $b$) gives us $( \bar A,\bar b)$.  This finishes the proof.
	
\end{proof}

From now on we will assume without loss of generality that all entries of $(A,b)$ are positive.
Formally, the uniform distribution on $ K_n$ can be written as
\begin{align}
P_n(dx)&=\frac{1}{Z_n}\d\lb Ax- b\rb\dopp1_{\R^{n}_{\geq0}}(x) dx\\
&=\frac{1}{Z_n}\d\lb A_{1\bullet}x- b_1\rb\cdot\dots\cdot\d\lb A_{m\bullet}x- b_m\rb \dopp1_{\R^{n}_{\geq0}}(x)dx,\label{definition_uniform}
\end{align}
where $\d$ denotes the ``Dirac delta function'' and $Z_n$ the normalizing constant. For a rigorous implementation of this idea, we will use Gaussian approximations of the delta functions. Define
\begin{align}
	P_{n}^{\g}(dx):=\frac{1}{Z_{n}^{\g}}e^{-\sum_{j=1}^{n}  w_j x_j-\g\sum_{l=1}^m\lb A_{l\bullet}x- b_l\rb^2}\dopp1_{\R^{n}_{\geq0}}(x)dx,\label{representation}
\end{align}
where $\g>0$ and $w=(w_1,\dots,w_{n})$ with $w_j>0$ for all $j$ is a so far arbitrary vector in $\Span\{A_{l\bullet},l=1,\dots,m\}$, and 
\begin{align}
	Z_{n}^\g:=\int_{\R^n_{\geq0}}e^{-\sum_{j=1}^{n}  w_j x_j-\g\sum_{l=1}^m\lb A_{l\bullet}x- b_l\rb^2}dx\label{normalizing_constant_gamma}
\end{align}
is the normalizing constant.	
\begin{prop}\label{prop:convergence}
As $\g\to\infty$ we have $P_n^\g\to P_n$ in distribution.
\end{prop}
\begin{proof}
Recall that a measurable set $E$ of the Borel $\s$-algebra $\mc B_{\R^n}$  is a continuity set of a (Borel) probability measure $P$ if $P(\partial E)=0$. We will show that for each continuity set $E\subset\R^n$ of $P_n$, we have $P_n^\g(E)\to P_n(E)$ as $\g\to\infty$, which is equivalent to convergence in distribution.	We start by a change of variables. $A$ has rank $m$ and we assume without loss of generality that its first $m$ columns are linearly independent. Then the matrix
	\begin{align}\label{def:M}
		M:=\begin{pmatrix}
			A & \\
			0 & I_{n-m}
		\end{pmatrix}
	\end{align}
has rank $n$, where $I_{n-m}$ denotes the $(n-m)\times(n-m)$ identity matrix.
	 The change of variables $y:=Mx$ corresponds to  setting $y_l:=A_{l\bullet}x,$ $l=1,\dots,m$ and $y_j:=x_j,\,j=m+1,\dots,n$. 
We have  
	\begin{align}
		P_n^\g(E)&=\frac{\int_{M(E\cap \R^n_{\geq0})}e^{-\sum_{j=m+1}^{n} w_j(M^{-1}y)_j-\sum_{l=1}^m\g\lb y_l- b_l\rb^2+w_l(M^{-1}y)_l}dy}{\int_{M\R^n_{\geq0}}e^{-\sum_{j=m+1}^{n} w_j(M^{-1}y)_j-\sum_{l=1}^m\g\lb y_l- b_l\rb^2+w_l(M^{-1}y)_l}dy}\\
		&=\frac{\int_{\R^{n-m}_{\geq0}}e^{-\sum_{j=m+1}^{n} w_j(M^{-1}y)_j}\lb\frac\g{\pi}\rb^{m/2}\int_{\R^m}\dopp1_{ M(E\cap \R^n_{\geq0})}(y)e^{-\sum_{l=1}^m\g\lb y_l- b_l\rb^2+w_l(M^{-1}y)_l}dy}{\int_{\R^{n-m}_{\geq0}}e^{-\sum_{j=m+1}^{n} w_j(M^{-1}y)_j}\lb\frac\g{\pi}\rb^{m/2}\int_{\R^m}\dopp1_{M\R^n_{\geq0}}(y)e^{-\sum_{l=1}^m\g\lb y_l- b_l\rb^2+w_l(M^{-1}y)_l}dy}.\quad\quad\label{eq_lemma_gamma}
	\end{align}
	Here, we used that $\sum_{j=m+1}^{n} w_j(M^{-1}y)_j$ does not depend on $y_1,\dots,y_m$.
		Looking at the numerator of \eqref{eq_lemma_gamma}, we see that we need to understand the limit $\g\to\infty$ of
	\begin{align}
		\lb\frac\g{\pi}\rb^{m/2}\int_{\R^m}\dopp1_{ M(E\cap \R^n_{\geq0})}(y)e^{-\sum_{l=1}^m\g\lb y_l- b_l\rb^2+w_l(M^{-1}y)_l}dy_1\dots dy_m.
	\end{align}
	Note first that due to the block structure of $M$, with $A=:(A',A'')$ with an $m\times m$ matrix $A'$, we have
	\begin{align}\label{M_inverse}
	    M^{-1}=\begin{pmatrix}
			A'^{-1} & -A'^{-1}A''\\
			0 & I_{n-m}
		\end{pmatrix}.
	\end{align}
	In the following we set for a vector $y\in\R^n$ $y=:(y'^{\,t},y''^{\,t})^t$ with $y'\in\R^m$ and $y''\in\R^{n-m}$. 
	Then we obtain
	\begin{align}
	    \sum_{l=1}^mw_l(M^{-1}y)_l=\sum_{l=1}^mw_l(A'^{-1}y')_l-\sum_{l=1}^mw_l(A'^{-1}A''y'')_l.
	\end{align}
	Completing the squares we get
\begin{align}
	-\sum_{l=1}^m\g\lb y_l- b_l\rb^2+\sum_{l=1}^mw_l(A'^{-1}y')_l=-\sum_{l=1}^m\g(y_l-b_{\g,\l}')^2+\frac{c_{l,1}}{\g}+\frac{c_{l,2}}{\g^2},
\end{align}
for some $b_{\g,l}'\in\R$ with $\lim_{\g\to\infty}b_{\g,l}'=b_l$, and $\g$-independent $c_{l,1}, c_{l,2}$. This shows
\begin{align}
	&\lb\frac\g{\pi}\rb^{m/2}\int_{\R^m}\dopp1_{M(E\cap\R^n_{\geq0})}(y)e^{-\sum_{l=1}^m\g\lb y_l- b_l\rb^2+\sum_{l=1}^mw_l(A'^{-1}y')_l}dy_1\dots dy_m\\
	&=\lb\frac\g{\pi}\rb^{m/2}\int_{\R^m}\dopp1_{M(E\cap\R^n_{\geq0})}(y)e^{-\sum_{l=1}^m\g\lb y_l- b_{\g,l}'\rb^2}dy_1\dots dy_me^{\frac{c_{l,1}}{\g}+\frac{c_{l,2}}{\g^2}},\label{equ:integral_convergence}
\end{align}
thus identifying the integral as (multiple of a) probability of a Gaussian random vector $(Y_1,\dots,Y_m)$ with mean vector $b'_\g$ and covariance matrix $\g^{-1}I$. It is thus obvious that as $\g\to\infty$, we have $(Y_1,\dots,Y_m)\to b$ in distribution. To use that convergence in the integral, it remains to be seen whether $\operatorname{Proj}_{m}(M(E\cap\R^n_{\geq0}))$ is a continuity set of $\d_b$, where $\operatorname{Proj}_{m}$ denotes projection on the first $m$ coordinates. We observe that $E'\subset\R^m$ is a continuity set of $\d_b$ if and only if $b$ is not in its boundary. If $(b^t,y''^{\, t})^t\in\partial M(E\cap \R^n_{\geq0})$ for some $y''\in\R^{n-m}_{\geq0}$, we have $M^{-1}(b^t,y''^{\, t})^t\in \partial E\cap K_n$.
	As $E$ is a continuity set of $P_n$, we have $P_n(\partial E)=0$, which means that $H^{n-m}(\partial E\cap K_n)=0$. This in turn implies that the set of $y''\in\R^{n-m}_{\geq0}$ such that $(b^t,y''^{\, t})^t\in \partial M(E\cap \R^n_{\geq0})$, has measure 0. Hence, we may assume in \eqref{equ:integral_convergence} that $(b^t,y''^{\, t})^t\notin\partial M(E\cap\R^n_{\geq0})$ for any $y''$.
 	Now we can use convergence in distribution to conclude that \eqref{equ:integral_convergence} converges for $\g\to\infty$ to
$	\dopp1_{S}(y'')$, where 
\begin{align}
&S:=\left\{y''\in\R^{n-m}_{\geq0}:\exists x\in E\cap \R^n_{\geq0} \text{ such that }Ax=b\text{ and } x''=y''\right\}\\
&=\operatorname{Proj}_{n-m}(M(E\cap K_n)),
\end{align}
and $\operatorname{Proj}_{n-m}(M(E\cap K_n))$ denotes, with an abuse of notation, the projection of $M(E\cap K_n)\subset \R^n$ on its last $n-m$ coordinates. Consequently the limit $\g\to\infty$ of the numerator of \eqref{eq_lemma_gamma} is
\begin{align}
	\int_{\operatorname{Proj}_{n-m}(M(E\cap K_n))}e^{-\sum_{j=m+1}^{n}w_j(M^{-1}(b^t,y''^{\, t})^t)_j-\sum_{l=1}^mw_l(A'^{-1}A''y'')_l}dy_{m+1}\dots dy_n,
\end{align}
where we used again that $\sum_{j=m+1}^{n}w_j(M^{-1}(y'^{\, t},y''^{\, t})^t)_j$ does not actually depend on $y'$ which we can therefore replace by $b$.
Now, by construction, $w$ is in the row span of $A$, i.e.~$w=(z^tA)^t$ for some $z\in\R^m$. Furthermore $(b^t,y''^{\, t})^t\in MK_n$, i.e.~$(b^t,y''^{\, t})^t=Mx$ for some $x\in K_n$.  This yields 
\begin{align}
	&\sum_{j=m+1}^{n}w_j(M^{-1}(b^t,y''^{\, t})^t)_j+\sum_{l=1}^mw_l(A'^{-1}A''y'')_l=\langle w,M^{-1}(b^t,y''^{\, t})^t\rangle-\sum_{l=1}^m w_lA'^{-1}b_l\\
	&=\langle z,Ax\rangle-\sum_{l=1}^m w_lA'^{-1}b_l=\langle z,b\rangle-\sum_{l=1}^m w_lA'^{-1}b_l,
\end{align}
hence the limit of the denominator of \eqref{eq_lemma_gamma} as $\g\to\infty$ is  
\begin{align}
    \exp\lb-\langle z,b\rangle+\sum_{l=1}^m w_lA'^{-1}b_l\rb\dopp{\l}^{n-m}(\operatorname{Proj}_{n-m}(MK_n)),
\end{align}
where $\mathbb{\l}^{n-m}$ denotes the $n-m$-dimensional Lebesgue measure. Analogously we find that the limit of the numerator of \eqref{eq_lemma_gamma} is
$\exp(-\langle z,b\rangle+\sum_{l=1}^m w_lA'^{-1}b_l)\l^{n-m}(\operatorname{Proj}_{n-m}(M(E\cap K_n) )$. Summarizing, we have
\begin{align}
	\lim_{\g\to\infty}P_n^\g(E)=\frac{\l^{n-m}(
		\operatorname{Proj}_{n-m}(M(E\cap K_n))}{\dopp{\l}^{n-m}(\operatorname{Proj}_{n-m}(MK_n))}.
\end{align}The result now follows from the area formula
\begin{align}
	H^{n-m}(T(E''))=\sqrt{\det(T^tT)}\l^{n-m}(E''), 
\end{align}
for any injective affine mapping $T:\R^{n-m}\to\R^n$ and all Borel sets $E''$ of $\R^{n-m}$ upon choosing 
\begin{align}
	T(y''):=M^{-1}(b^t,y''^{\, t})^t,\quad y''\in\R^{n-m}
\end{align}
as this shows that 
\begin{align}
	\lim_{\g\to\infty}P_n^\g(E)=\frac{H^{n-m}(E\cap K_n)}{H^{n-m}(K_n)}=P_n(E)\label{end_proof}
\end{align}
by \eqref{def:Hausdorff}.
\end{proof}

\section{The complex de Finetti representation and a Bartlett type formula}\label{sec:deFinetti}
An advantage of the at first sight complicated approach to the uniform distribution on $K_n$ via $P_n^\g$ is the possibility to use the more accessible $P_n^\g$ for a representation of $P_n$. We first prove Theorem \ref{thrm_mixture} in two steps. In the first step we exploit the probabilistic structure of $P_n^\g$ to derive the mixture representation \eqref{equ:mixture} under a technical assumption. In the second step we link that technical assumption to the rank of $A$.

\begin{prop}
\label{prop_mixture}
	Let $P_n$ denote the uniform distribution on $ K_n$ and $P_{n,w}$ the distribution on $\R^{n}$ with density
	\begin{align}
	\lb\prod_{j=1}^{n}w_j\rb e^{-\sum_{j=1}^{n}  w_jx_j}\dopp1_{\R^{n}_{\geq0}}(x),
	\end{align}
	where  $w\in\Span\{A_{l\bullet},l=1,\dots,m\}$ is arbitrary subject to $w_j>0$ for all $j$.
	Assume the integrability condition 
	\begin{align}
	\eta\mapsto\E_{P_{n,w}}e^{\i \sum_{j=1}^n\langle \eta,A_{\bullet j}\rangle X_j}\in L^1(\R^m,d\eta).\label{integrability-condition}
	\end{align}
	Then, for any continuity set $E$ of $P_n$  of the form $E=\bigtimes_{j=1}^n E_j$ with the $E_j$ being intervals in $\R$ we have
	\begin{align}\label{equ:mixture2}
        P_n(E)=\int_{\R^m} P_{n,w,\eta}(E)dP_{n,\textup{mix}}(\eta),
    \end{align}
    where $P_{n,w,\eta}$ and $P_{n,\textup{mix}}$ are defined as in Theorem \ref{thrm_mixture}.
\end{prop}

\begin{proof}
	The proof uses the identity, known as Hubbard-Stratonovich transform,
	\begin{align}
	e^{-\g\lb A_{l\bullet}x- b_l\rb^2}=\frac{1}{\sqrt{4\pi\g}}\int_\R e^{\i\eta_l\lb A_{l\bullet}x- b_l\rb}e^{-\frac{1}{4\g}\eta_l^2}d\eta_l.\label{HS-transform}
	\end{align}
	Applying \eqref{HS-transform} to \eqref{representation} gives
	\begin{align}
	 P_n^\g(E)=\frac{\int_{\R^m}\int_{E\cap\R^n_{\geq0}} e^{-\sum_{j=1}^nw_jx_j+\i\sum_{l=1}^m\eta_l\lb A_{l\bullet}x- b_l\rb}dxe^{-\sum_{l=1}^m\frac{1}{4\g}\eta_l^2}d\eta}{\int_{\R^m}\int_{\R^n_{\geq0}} e^{-\sum_{j=1}^nw_jx_j+\i\sum_{l=1}^m\eta_l\lb A_{l\bullet}x- b_l\rb}dxe^{-\sum_{l=1}^m\frac{1}{4\g}\eta_l^2}d\eta}.\label{pre-Bartlett}
	\end{align}
	 Using
	\begin{align}\label{equ:conversion}
	 \i\sum_{l=1}^m\eta_l\lb A_{l\bullet}x- b_l\rb=\i\sum_{j=1}^nx_j\langle\eta,A_{\bullet j}\rangle-\i\langle\eta,b\rangle,
	\end{align}
	and observing that 
	\begin{align}\label{identity_complex_normalizing}
	Z_{n,w,\eta}=Z_{n,w}\E_{P_{n,w}}e^{\i\sum_{j=1}^n\langle\eta,A_{\bullet j}\rangle X_j}=Z_{n,w}\prod_{j=1}^n\frac{w_j}{w_j-\i\langle\eta,A_{\bullet j}\rangle}\not=0,
	\end{align}
	we get
	\begin{align}
	 P_n^\g(E)&=\frac{\int_{\R^m}\int_{E\cap\R^n_{\geq0}} e^{-\sum_{j=1}^n(w_j-\i\langle\eta,A_{\bullet j}\rangle)x_j}dxe^{-\i\langle\eta,b\rangle-\sum_{l=1}^m\frac{1}{4\g}\eta_l^2}d\eta}{\int_{\R^m}\int_{\R^n_{\geq0}} e^{-\sum_{j=1}^n(w_j-\i\langle\eta,A_{\bullet j}\rangle)x_j}dxe^{-\i\langle\eta,b\rangle-\sum_{l=1}^m\frac{1}{4\g}\eta_l^2}d\eta}\label{pre-mixture2}\\
	 &=\frac{\int_{\R^m}P_{n,w,\eta}(E)Z_{n,w,\eta}e^{-\i\langle\eta,b\rangle-\sum_{l=1}^m\frac{1}{4\g}\eta_l^2}d\eta}{\int_{\R^m}Z_{n,w,\eta}e^{-\i\langle\eta,b\rangle-\sum_{l=1}^m\frac{1}{4\g}\eta_l^2}d\eta}\label{pre-mixture3}\\
	 &=\frac{\int_{\R^m}P_{n,w,\eta}(E)\E_{P_{n,w}}e^{\i \sum_{j=1}^n\langle \eta,A_{\bullet j}\rangle X_j-\i\langle\eta,b\rangle-\sum_{l=1}^m\frac{1}{4\g}\eta_l^2}d\eta}{\int_{\R^m}\E_{P_{n,w}}e^{\i \sum_{j=1}^n\langle \eta,A_{\bullet j}\rangle X_j-\i\langle\eta,b\rangle-\sum_{l=1}^m\frac{1}{4\g}\eta_l^2}d\eta},\label{pre-mixture4}
	\end{align}
	where for the last equality we divided both numerator and denominator by $Z_{n,w}$ and used
	\begin{align}
	    \frac{Z_{n,w,\eta}}{Z_{n,w}}=\E_{P_{n,w}}e^{\i \sum_{j=1}^n\langle \eta,A_{\bullet j}\rangle X_j}.
	\end{align}
	Taking the limit $\g\to\infty$ on the l.h.s.~of \eqref{pre-mixture4} yields $P_n(E)$ and thus the l.h.s.~of \eqref{equ:mixture2} by Proposition \ref{prop:convergence}. 	
	
	To address taking the limit $\g\to\infty$ of the r.h.s.~of \eqref{pre-mixture4}, we first note that by our assumption \eqref{integrability-condition}, the limit $\g\to\infty$ of the denominator in \eqref{pre-mixture4} can be taken inside of the $\eta$-integral. We now observe that the denominator in \eqref{pre-mixture4} is just a multiple of the denominator in \eqref{eq_lemma_gamma}, which we have shown in \eqref{end_proof} to converge to $H^{n-m}(K_n)>0$. Thus in particular the limit of the denominator in \eqref{pre-mixture4} is non-zero, justifying the seperate consideration of numerator and denominator. It remains to show that we can interchange the limit $\g\to\infty$ with the $\eta$-integral in the numerator of \eqref{pre-mixture4}. Comparing with the denominator, we see that we need to show boundedness  of  $P_{n,w,\eta}(E)$ in $\eta$ (recall that $P_{n,w,\eta}(E)\in\C$). From \eqref{def:complex_measure} we find 
	\begin{align}
	   P_{n,w,\eta}(E)=\frac{\E_{P_{n,w}}\dopp1_{E}(X)e^{\i\sum_{j=1}^n\langle\eta,A_{\bullet j}\rangle X_j}}{\E_{P_{n,w}}e^{\i\sum_{j=1}^n\langle\eta,A_{\bullet j}\rangle X_j}}.	\end{align}As we saw in \eqref{identity_complex_normalizing}, we have 
	\begin{align}\label{denominator_decay}
\E_{P_{n,w}}e^{\i\sum_{j=1}^n\langle\eta,A_{\bullet j}\rangle X_j}=\prod_{j=1}^n\frac{w_j}{w_j-\i\langle\eta,A_{\bullet j}\rangle}
	\end{align}
	and integrability in $\eta$ of this expression is our assumption \eqref{integrability-condition}.
Let now without loss of generality $E\cap\R^n_{\geq0}=\bigtimes_{j=1}^n (a_j,b_j)$, then we readily compute
	\begin{align}
\E_{P_{n,w}}\dopp1_{E}(X)e^{\i\sum_{j=1}^n\langle\eta,A_{\bullet j}\rangle X_j}=\prod_{j=1}^n\frac{w_j\lb e^{-b_j(w_j-\i\langle\eta,A_{\bullet j}\rangle}-e^{-a_j(w_j-\i\langle\eta,A_{\bullet j}\rangle}\rb}{w_j-\i\langle\eta,A_{\bullet j}\rangle},
	\end{align}
which has the same decay in $\eta$ as \eqref{denominator_decay} and hence $P_{n,w,\eta}(E)$ is bounded in $\eta$.
\end{proof}
The following lemma links the integrability condition \eqref{integrability-condition} to $A$ and $K_n$. In particular it shows that \eqref{integrability-condition} holds in all non-degenerate cases, in the sense that no variable is constant.

\begin{lemma}\label{lemma_integrability}
	The following statements are equivalent:
	\begin{enumerate}
		\item The integrability condition \eqref{integrability-condition} holds.
		\item The rank of $A$ can not be reduced by removing one column.
		\item For any $j=1,\dots,n$, the interval $P_jK_n$ has a non-empty interior, where $P_j:\R^n\to\R, P_j(x):=x_j$, denotes the projection onto the $j$-th coordinate. 
	\end{enumerate}
\end{lemma}
\begin{proof}
	We first prove that (1) implies (2) by showing that if the rank of $A$ can be reduced by removing one column, then \eqref{integrability-condition} fails. So assume that (2) does not hold. Then without loss of generality we may assume that $A_{\bullet n}=(0,0,\dots,0,1)^t\in\R^m$ and $A_{mj}=0$ for $j=1,\dots,n-1$, otherwise this situation can be generated by a linear transformation that leaves $K_n$ unchanged, and a relabeling of variables. Now, analogously to \eqref{identity_complex_normalizing} we get
	\begin{align}
	\int_{\R^m}\lvv\E_{P_{n,w}} e^{\i\sum_{l=1}^m\eta_l\lb A_{l\bullet}X- b_l\rb}\rvv d\eta=\int_{\R^m}\prod_{j=1}^n\frac{ w_j}{\sqrt{ \lb w_j\rb^2+\lb\sum_{l=1}^m\eta_l A_{lj}\rb^2}}d\eta.
	\end{align}
	By the shape of $A_{\bullet n}$, the $n$-th factor gives a decay in $\eta_m$ of order $\O(\lv\eta_m\rv^{-1})$ as $\eta_m\to\pm\infty$. Moreover, since $A_{mj}=0$ for $j\not=n$, the other factors do not decay in $\eta_m$. Hence, \eqref{integrability-condition} does not hold.
	
	To prove that (2) implies condition \eqref{integrability-condition}, we use the Brascamp-Lieb inequality.  Let $c\in\R^n_{\geq0}$ with $\frac1m\sum_{j=1}^n c_j=1$ be given, as well as vectors $B_j\in\R^m$, $B_j\not=0$, $j=1,\dots,n$, that span $\R^m$. Then the (univariate) Brascamp-Lieb inequality asserts that for all functions $f_j\in L^1(\R)$, $f_j\geq0$, $j=1,\dots,n$, the inequality
	\begin{align}
	\int_{\R^m}\prod_{j=1}^n f_j^{c_j}(\eta B_j)d\eta\leq \frac1{\sqrt{D_{c,B}}}\prod_{j=1}^n\lb\int_\R f_j(x)dx\rb^{c_j}\label{BL-inequality}
	\end{align}
	holds, where (the vectors $B_j$ being understood as column vectors)
	\begin{align}
	D_{c,B}:=\inf\left\{\frac{\det\lb\sum_{j=1}^n\s_jB_jB_j^t\rb}{\prod_{j=1}^n\s_j^{c_j}},\ \s_j>0,\,j=1,\dots,n\right\}.\label{BL-constant}
	\end{align}
	The constant $D_{c,B}$ is positive or zero. Barthe has in \cite{Barthe}  given the following nice geometric characterization of the cases where $D_{c,B}$ is positive, which remarkably does not depend on $B_j$, $j=1,\dots,n$. To state it, let for $I\subset\{1,\dots,n\}$, $\mb1_I$ denote the vector $(\mb1_I)_j:=1$ if $j\in I$ and 0 else. Now, \cite[Proposition 3]{Barthe} shows that $D_{c,B}>0$ if and only if $c\in\R^n_{\geq0}$ belongs to the convex hull of the vectors $\mb1_I$ of the subsets $I\subset \{1,\dots,n\}$, $\lv I\rv=m$ such that $B_j,j\in I$ span $\R^m$.
	
	In our application of \eqref{BL-inequality}, we take $B_j:=A_{\bullet j}$.	
By assumption, the rank of $A$ is $m$ and thus the columns $A_{\bullet j}$ span $\R^m$.  Moreover, $A_{\bullet j}\not=0$ for all $j$, since otherwise $K_n$ would not be bounded. This together with assertion (2) of the lemma implies that for each $j=1,\dots,n$, there are subsets $I_{j,1},I_{j,2}\subset\{1,\dots,n\}$ with $\lv I_{j,1}\rv=\lv I_{j,2}\rv=m$ such that $j\in I_{j,1},\,j\notin I_{j,2}$, and  $(A_{\bullet j},j\in I_i)$ spans $\R^m$, $i=1,2$. Defining 
\begin{align}
c:=\frac{1}{2n}\sum_{j=1}^n \mb1_{I_{j,1}}+\mb 1_{I_{j,2}}     
\end{align}
we get $0<c_j<1$ for all $j$, hence  $D_{c,B}>0$ in \eqref{BL-inequality}.
Defining the functions $f_j$ by
	\begin{align}
	f_j^{c_j}(y):=\frac{ w_j}{\sqrt{ \lb w_j\rb^2+y^2}},
	\end{align}
	it follows that $f_j\in L^1(\R)$. Now (1) follows from \eqref{BL-inequality}.
	
	To address the equivalence of (3) and (2), we observe that $X_j$ being constant on $K_n$ is equivalent to it being fixed by a constraint $A_{l\bullet}X=b_l$ of the form (possibly after a changing to an equivalent representation $(A,b)$) $A_{l\bullet}=(0,\dots,0,1,0,\dots,0)$ with the $1$ being at position $j$. As above, we may also assume that $A_{mj}=0$ for $m\not=l$, which implies that removing column $j$ reduces the rank of $A$ and hence (2) does not hold. On the other hand, if removing a column, say  $A_{\bullet j}$, reduces the rank of $A$, then the reduced row echelon form of $A$ features a row $A_{l\bullet}=(0,\dots,0,1,0,\dots,0)$ with the $1$ at position $j$, which is as per the above equivalent to $X_j$ being constant.
\end{proof}

\begin{proof}[Proof of Theorem \ref{thrm_mixture}]
The theorem follows directly from combining Proposition \ref{prop_mixture} and Lemma \ref{lemma_integrability}. Note in particular that part 3 of Lemma \ref{lemma_integrability} implies that $E$ is a continuity set for $P_n$ so that the assumption in Proposition \ref{prop_mixture} can be dropped.
\end{proof}

For the proofs of our main results we will apply Theorem \ref{thrm_mixture} to the characteristic function of $S_n$. 
\begin{cor}
\label{cor_Bartlett}
Assume that the rank of $A$ can not be reduced by removing one column.
	Then, for any $t\in\R$,
	\begin{align}
	\E_{P_n}e^{\i t S_n}=\frac{\int_{\R^m}\E_{P_{n,w}} e^{\i t S_n+\i\sum_{l=1}^m\eta_l\lb A_{l\bullet}X- b_l\rb}d\eta}{\int_{\R^m}\E_{P_{n,w}} e^{\i\sum_{l=1}^m\eta_l\lb A_{l\bullet}X- b_l\rb}d\eta},\label{Bartlettintro}
	\end{align}
	where $\E_{P_{n}}$ and $\E_{P_{n,w}}$ denote expectation w.r.t.~$P_n$ and $P_{n,w}$, respectively. 
\end{cor}
Note that in \eqref{Bartlettintro}, expectations are always taken w.r.t.~independent random variables, thus converting the problem of proving a CLT under geometric dependence to proving a CLT for independent random variables. Formula \eqref{Bartlettintro} may be seen as a generalization of an old formula due to Bartlett (see \cite{Bartlett}) for the characteristic function of a random variable conditioned on the value of another random variable.
 \begin{proof}
	We follow the same strategy as in the proof of Proposition \ref{prop_mixture}. Instead of \eqref{pre-Bartlett} and \eqref{pre-mixture4} we get
	\begin{align}
	\E_{P_n^\g}e^{\i t S_n}=\frac{\int_{\R^m}\E_{P_{n,w}} e^{\i t S_n+\i\sum_{l=1}^m\eta_l\lb A_{l\bullet}X- b_l\rb}e^{-\sum_{l=1}^m\frac{1}{4\g}\eta_l^2}d\eta}{\int_{\R^m}\E_{P_{n,w}} e^{\i\sum_{l=1}^m\eta_l\lb A_{l\bullet}X- b_l\rb}e^{-\sum_{l=1}^m\frac{1}{4\g}\eta_l^2}d\eta}.\label{pre-Bartlett_cf}
	\end{align}
	Taking the limit $\g\to\infty$ on the l.h.s.~of \eqref{pre-Bartlett_cf} yields the l.h.s.~of \eqref{Bartlettintro}. For the convergence of the r.h.s.~of \eqref{pre-Bartlett}, we  show that the limit $\g\to\infty$ can be taken inside of the two integrals in numerator and denominator. To see the integrability of numerator and denominator in \eqref{pre-Bartlett_cf} we observe
	\begin{align}
	&\lv\E_{P_{n,w}} e^{\i t S_n+\i\sum_{l=1}^m\eta_l\lb A_{l\bullet}X- b_l\rb}\rv=\lv\E_{P_{n,w}} e^{\i t \sum_{j=1}^n \l_jX_j+\i\sum_{l=1}^m\eta_l A_{l\bullet}X}\rv\\
	&=\lv\E_{P_{n,w}} e^{\i\sum_{j=1}^nX_j\lb\sum_{l=1}^m\eta_l A_{lj}+t\l_j\rb}\rv\\
	&=\prod_{j=1}^n\frac{ w_j}{\sqrt{ \lb w_j\rb^2+\lb\sum_{l=1}^m\eta_l A_{lj}+t\l_j\rb^2}}\leq C_n \prod_{j=1}^n\frac{ w_j}{\sqrt{ \lb w_j\rb^2+\lb\sum_{l=1}^m\eta_l A_{lj}\rb^2}}\label{equ:prooflemma1}\\
	&=C_n\lv\E_{P_{n,w}} e^{\i\sum_{l=1}^m\eta_lA_{l\bullet}X}\rv,
	\end{align}
	where $C_n>0$ is chosen so large that we have 
	\begin{align}
	\lb w_j\rb^2+\lb\sum_{l=1}^m\eta_l A_{lj}+t\l_j\rb^2\geq \frac1{C_n^2}\lb \lb w_j\rb^2+\lb\sum_{l=1}^m\eta_l A_{lj}\rb^2\rb
	\end{align}
	for any $j$. By Lemma \ref{lemma_integrability} and  \eqref{integrability-condition}, we conclude that both numerator and denominator are integrable, uniformly in $\g$. As we saw in the proof of Proposition \ref{prop_mixture} already, the denominator of \eqref{pre-Bartlett_cf} does not converge to 0 as $\g\to\infty$. This implies that the r.h.s.~of \eqref{pre-Bartlett_cf} converges as $\g\to\infty$ to the r.h.s.~of \eqref{Bartlettintro}.

 \end{proof}

\section{Maximum entropy principle for the probabilistic barycenter}
\label{su_maxent}
In this section we will deal with the choice of the so far not determined weight vector $w$. For the two previous sections we only needed  $w\in\Span\{A_{l\bullet},l=1,\dots,m\}$ along with positivity of each of its elements. For taking the limit $n\to\infty$ it will however turn out to be very convenient to have the statistics $ A_{l\bullet}X-  b_l$, $l=1,\dots,m$ centered under $P_{n,w}$, i.e.~we seek $w$ such that 
\begin{align}
\E_{P_{n,w}} A_{l\bullet}X=  b_l,\quad l=1,\dots,m.\label{centering-condition1}
\end{align}
From the definition of $P_{n,w}$, it is apparent that \eqref{centering-condition1} can be written as
\begin{align}
A_{l\bullet}\frac{1}{w}=  b_l,\quad l=1,\dots,m,\label{centering-condition2}
\end{align}
where $\frac{1}{w}$ stands for the (column) vector $\lb\frac{1}{w}\rb_j:=\frac{1}{w_j}$, $j=1,\dots,n$. Hence, we must have $\frac{1}{w}\in K_n$. The next proposition shows that there is one and only one such $w$. Moreover, it enjoys a nice variational characterization.

\begin{prop}\noindent\label{Proposition:w}
There is a unique $w>0$ with $w\in\Span\{A_{l\bullet},l=1,\dots,m\}$ such that \eqref{centering-condition2} holds. Moreover, $P_{n,w}$ (with this special $w$) has maximum entropy among the distributions  on $\R^n_{\geq0}$ with expectation vector lying in $K_n$. Furthermore, $w^t=A^ t\Lambda_0$, where $\Lambda_0$
is the only minimizer of the convex function
\begin{align}\label{Edual}
H_n(\Lambda)&:=\log\int_{\R^n_{\geq0}}\exp(-\Lambda^tAx)dx
+\langle\Lambda,b\rangle,\\
& =  -\sum_{j=1}^n\log[(A^t\Lambda)_j]+\langle\Lambda,b\rangle.
\end{align}
Here, $H_n$ may be infinite and we set $\log\tau=-\infty$ whenever 
$\tau\leq 0$.
\end{prop}
Before giving the proof of the previous proposition, let us recall that
the entropy $S(P)$ of an absolutely continuous probability measure $P$ on $\R^n_{\geq0}$ with density $f$ is 
defined by 
$$ S(P):= \begin{cases}
-\displaystyle\int_{\R^n_{\geq0}}f(x)\log f(x) dx,\;\; \mbox{ if $f\log f$ is integrable,}\\
-\infty,\;\; \mbox{ otherwise.}
\end{cases}$$
Furthermore, if $P$ and $Q$ are probability measures on $\R^n_{\geq0}$, the Kullback-Leibler divergence or cross entropy of $P$ with respect to $Q$ is given by  
$$ K(P,Q)= \begin{cases}
\displaystyle\int_{\R^n_{\geq0}}\log\lb\frac{dP}{dQ}\rb P(dx),\;\; \mbox{ if $P\ll Q$ and $\frac{dP}{dQ}\in L^1(P)$},\\
+\infty,\;\; \mbox{ otherwise.}
\end{cases}$$

\begin{proof}
From Lemma \ref{lemma-positivity}, we may assume that the entries of $A$ are positive. Hence, denoting as before by $\mathbf{1}$ the vector whose all entries equal $1$, we may assert that $P_{n,A^t\mathbf{1}}$ is well-defined.  Let $Q$ be an absolutely continuous probability measure on $\R^n_{\geq 0}$ whose density has finite entropy and whose expectation vector lies in $K_n$. Then,
\begin{equation}
K(Q,P_{n,A^t\mathbf{1}})=H_n(\mathbf{1})-S\left(Q\right).
\label{Eentropic}
\end{equation}
Hence, maximizing $S$ over such $Q$ means minimizing 
$K(\cdot,P_{n,A^t\mathbf{1}})$.
Now, as the relative interior of $K_n$ is non-empty by assumption, there exists a vector $\bar x$ with $\bar x_j>0$ for each $j$
satisfying $A\bar x= b$. Let  $\mathcal{E}(\tau)$ denote the exponential distribution with mean $\tau>0$. Setting $Q:=\otimes_{j=1}^{n}\mathcal{E}(\bar x_{j})$ yields a probability measure on $\R^n_{\geq0}$ with expectation vector lying in $K_n$. Furthermore, 
 $Q$ and $P_{n,A^t\mathbf1}$  are mutually absolutely continuous. Hence,  as the set where 
$H_n$ is finite is  non empty and  open, we may conclude applying  directly
\cite[Theorem 2 and corollary]{csiszar1984sanov}.
\end{proof}
The following corollary follows directly from equation \eqref{Eentropic} and the Pythagorean inequality for the Kullback-Leibler information (see 
\cite[Theorem 2.2]{csiszar1975divergence}).
\begin{cor}
\label{Cor_Entrop}
With the notations of the previous proposition, the vector $(1/w_j)_{j=1,\dots,n}$ is the only minimizer over $K_n$ of the function 
$-\sum_{j=1}^n\log(x_j)$.
\end{cor}
\section{Proof of Proposition \ref{prop:examplebis}}\label{sec:proofexample}
As we will need Property A, see \eqref{det-condition}, for the proof of Theorem \ref{thrm1}, we will provide a proof of Proposition \ref{prop:examplebis} in this section, which covers Property A and also shows the crucial genericity of our assumption. 

The proof of part (2) of the proposition is obvious by symmetry arguments.
The proof of part (1) of Proposition \ref{prop:examplebis}  relies on the following result, which is a key ingredient of a proof of the L\'evy-Steinitz Theorem (a multivariate generalisation of Riemann's Rearrangement Theorem).

\begin{lemma}[{\cite[Lemma 2]{Rosenthal87}}]\label{lemma_Levy-Steinitz}
	If $\{v_j\}_{j=1}^n\subset\R^d$, $w=\sum_{j=1}^n v_j$, $0<t<1$, and $\|v_j\|\leq \e$ for all $j$, then either $\|v_1-tw\|\leq C_d\e$ for some constant $C_d>0$ only depending on $d$, or there is a permutation $\pi$ of $(2,\dots,n)$ and an $r$ between 2 and $n$ such that $\|v_1+\sum_{j=2}^rv_{\pi(j)}-tw\|\leq C_d\e$ for the same $C_d$.
\end{lemma}

\begin{proof}[Proof of Proposition  \ref{prop:examplebis} (1)]
	We will first show that we can choose for any $\e>0$ small enough $I_{1},\dots,I_{K}$ pairwise disjoint such that
	\begin{align}
		\lb\hat A_{I_l}\hat A_{I_l}^t\rb_{ii}\geq\frac1{2K}\quad \text{ and } \quad \lvv\lb\hat A_{I_l}\hat A_{I_l}^t\rb_{ij}\rvv\leq\e,\quad i\not=j,\label{near-diagonal}
	\end{align}
	in other words such that $\hat A_{I_l}\hat A_{I_l}^t$ is close to a diagonal matrix with large enough diagonal terms. 
	
	Let $I\subset\{1,\dots,n\}$ and $a_{ij}:=a_{ij}(I):=\lb\hat A_{I}\hat A_{I}^t\rb_{ij}$. To explain the idea of the proof, note that $a_{ii}\in[0,1]$, and $a_{ii}$ is non-decreasing when adding more columns to $I$.
	Because of $\|\hat A\|_{\max}=o(1)$, adding a column to $I$ results in marginal changes to  $\hat A_{I}\hat A_{I}^t$. Since $\langle\hat A_{l'\bullet},\hat A_{l''\bullet}\rangle=0$ and again $\|\hat A\|_{\max}=o(1)$, $a_{ij}$ should for $i\not=j$ be seen as an alternating series as $n\to\infty$. It is clear that it is possible to increase $I$ and thereby $a_{ii}$ for all $i$ while keeping $\lv a_{ij}\rv$ small \emph{for a fixed pair} $i\not=j$, by alternating between selecting columns corresponding to non-negative and negative contributions to $a_{ij}$, respectively. This is in the spirit of Riemann's theorem on rearrangements of conditionally convergent real series. We however need to control the size of all entries $a_{ij}$ \emph{simultaneously}. Here, we use the lesser-known multivariate extension of Riemann's theorem named after L\'evy and Steinitz in the form of Lemma \ref{lemma_Levy-Steinitz}. We use 
	\begin{align}
		v_j:=(\hat A_{ij}\hat A_{kj})_{1\leq i,k\leq m},\quad 1\leq j\leq n,
	\end{align}
	 which we consider as vectors in $\R^{m^2}$. We have 
	 \begin{align}
	 	\sum_{j=1}^n v_j=\hat A\hat A^t=I.
	 \end{align}
	As $\|\hat A\|_{\max}=o(1)$ we get $\|v_j\|=o(1)$ uniformly in $j$ and for $t:=\frac1{K+1}$ and $w:=w_1:=I$ we can, by Lemma \ref{lemma_Levy-Steinitz}, find for any $\e>0$ a permutation $\pi$ and $r$ such that $\|v_1+\sum_{j=2}^rv_{\pi(j)}-tI\|\leq \e$. This gives us a set of columns $I_1:=\{1,\pi(2),\dots,\pi(r)\}$ such that \eqref{near-diagonal} holds for $l=1$, for $\e>0$ small enough. We can now iterate this procedure, where in the next iteration $w_2:=I-t\hat A_{I_1}\hat A_{I_1}^t$ is nearly diagonal.
	
	After $K$ iterations we arrive at pairwise disjoint sets $I_1,\dots,I_{K}$ such that \eqref{near-diagonal} holds. Now, by Gershgorin's circle theorem, the eigenvalues of  $\hat A_{I_l}\hat A_{I_l}^t$ are bounded below by 
	\begin{align}
		\frac1{2K}-(K-1)\e
	\end{align}
	and choosing $\e$ small enough shows that they are bounded away from 0 uniformly in $n$. This implies part 1 of the proposition.
\end{proof}

\begin{proof}[Proof of Proposition  \ref{prop:examplebis} (3)]
This proof will use classical results on $Z$-estimation (see for example Chapter 5 in \cite{vandervaart}). 
First recall that
$\Lambda_0$ from Proposition \ref{Proposition:w} is defined as the only minimizer of the function  
$$ H_n(\Lambda)=-\sum_{j=1}^n\log[(A^{t}\Lambda)_j]+\langle\Lambda,b\rangle,$$
provided existence.
Here, we set $\log \theta:=-\infty$ whenever $\theta\leq 0$.
Hence $\Lambda_0$ is the only vector satisfying the equation
$$-\sum_{j=1}^n \frac{A_{\bullet j}}{\langle\Lambda,A_{\bullet j}\rangle}+b=0.$$
Setting $v:=\frac{\Lambda}{n}$, 
$v_0:=\frac{\Lambda_0}{n}$ is the unique zero of the function
$$-\frac{1}{n}
\sum_{j=1}^{n}
\frac{A_{\bullet j}}{\langle v,A_{\bullet j}\rangle}+b .$$
But, for $v\in\inte(\dom G)$ the last function converges  in probability to
$\nabla G(v)+b$ as $n\to\infty$, whereas for $v\not\in\dom G$ it diverges in probability to $+\infty$. Now, \cite[Theorems 5.41 and 5.42, p. 68]{vandervaart} allow to conclude. Indeed, here for $v\in\inte(\dom G)$,
$$\nabla^2G(v)=-\nabla\int_U\frac{u P(du)}{\langle v, u\rangle}= \int_U\frac{u u^t P(du)}{\langle v, u\rangle}=:\Sigma(b).$$
This last matrix is invertible as $\mathcal{U}$ (the closed convex hull of $U$),  is not a subset of any hyperplane. Hence, we may conclude that
${\displaystyle\sqrt{n}\left(\frac{\Lambda_0}{n}-v(b)\right)}$ converges in distribution towards a centered Gaussian vector with covariance matrix $\Sigma^{-1}(b)$. 
So, using the definition of $\tilde A$ in \eqref{def:Atilde} and $w=A^t\Lambda_0$ from Proposition \ref{Proposition:w}, we may write, for $i=1,\dots, m,j=1,\dots ,n$,
$$n\tilde{A}_{ij}=\frac{A_{ij}}{\langle v(b),A_{\bullet j}\rangle}+o_p(1).$$
In particular, $\tilde{A}_{ij}=\O_p(n^{-1})$ thanks to $A_{\bullet j}$ being bounded and bounded away from $0$. Moreover, for $i,j=1,\dots, m$,
$$n(\tilde{A}\tilde{A}^t)_{ij}=\frac{1}{n}\sum_{l=1}^n\frac{A_{il}A_{jl}}{\langle v(b),A_{\bullet l}\rangle^2}+o_p(1). $$
Here, the $o_p$ are uniform both in the  last two equations. Thus the last equality implies  
$$\lim_{n\rightarrow\infty}\sqrt{\frac{1}{n}}(\tilde{A}\tilde{A}^t)^{-\frac{1}{2}}=\Sigma(b)^{-\frac{1}{2}}.$$
Hence, $(\tilde{A}\tilde{A}^t)^{-\frac{1}{2}}=\O_p(\sqrt n)$ and with $\tilde{A}_{ij}=\O_p(n^{-1})$ as seen above we conclude that $\|\hat{A}\|_{\max}=o_p(1)$. This completes the proof.
\end{proof}

\section{Asymptotic analysis}
\label{su_AsA}
In this section we will prove Theorems \ref{thrm1} and \ref{thrm_marginal}. From now on let $w$ be the unique entropy-maximizing $w$ found in Proposition \ref{Proposition:w}. The starting point is Corollary \ref{cor_Bartlett}. In \eqref{Bartlettintro}, $S_n$ is centered w.r.t.~$P_n$. It is however more convenient to first prove the CLT in Theorem \ref{thrm1} under $S_n$ being centered w.r.t.~$P_{n,w}$ as all computations can be carried out using $P_{n,w}$. To this end, define
\begin{align}
	S_n^*:=\sum_{j=1}^n \l_j(X_j-\E_{P_{n,w}} X_j)\label{def:S_n*}
\end{align}
and note that \eqref{Bartlettintro} also holds with $S_n^*$ replacing $S_n$. We further observe that we have 
\begin{align}
	&\E_{P_n}e^{\i t S_n^*}
	=\frac{\int_{\R^m}\E_{P_{n,\mb 1}} \exp\lb{\i t \hat S_n+\i\sum_{l=1}^m\eta_l\lb \hat A_{l\bullet}X-   \hat b_l\rb}\rb d\eta}{\int_{\R^m}\E_{P_{n,\mb 1}}\exp\lb{\i\sum_{l=1}^m\eta_l\lb \hat A_{l\bullet}X-   \hat b_l\rb}\rb d\eta},\label{Bartlett2}
\end{align}
which is now a representation in terms of i.i.d.~exponential random variables with parameter 1, i.e.~in terms of $P_{n,\mb 1}$ with $\mb 1=(1,\dots,1)$, and 
\begin{align}
	\hat S_n:=\sum_{j=1}^n\hat\l_j(X_j-1).
\end{align}

\begin{proof}[Proof of Theorem \ref{thrm1}]
Changing variables from $\eta$ to $\s^{-1}\eta$ in numerator and denominator in \eqref{Bartlett2}, we get
\begin{align}
	&\E_{P_n}e^{\i\s^{-1} t S_n^*}
	=\frac{\int_{\R^m}\E_{P_{n,\mb 1}} \exp\lb{\i\s^{-1} t \hat S_n+\i\s^{-1}\sum_{l=1}^m\eta_l\lb \hat A_{l\bullet}X-   \hat b_l\rb}\rb d\eta}{\int_{\R^m}\E_{P_{n,\mb 1}}\exp\lb{\i\s^{-1}\sum_{l=1}^m\eta_l\lb \hat A_{l\bullet}X-   \hat b_l\rb}\rb d\eta}.\label{Bartlett2'}
\end{align}
We will first analyze the asymptotics of
 \begin{align}
		&\E_{P_{n,\mb 1}} \exp\lb{\i \s^{-1}t \hat S_n+\i\s^{-1}\sum_{l=1}^m\eta_l\lb \hat A_{l\bullet}X-  \hat b_l\rb}\rb\\
		&=\E_{P_{n,\mb1}}\exp\lb\i\s^{-1}\lb\sum_{j=1}^n X_j\lb t\hat\l_j+\langle\eta,\hat A_{\bullet j}\rangle\rb-\sum_{j=1}^n \lb t\hat\l_j+\langle\eta,\hat A_{\bullet j}\rangle\rb\rb\rb\label{equ:startproof}
	\end{align}
	for fixed $\eta$ using cumulants, where we used for \eqref{equ:startproof} the identity
	\begin{align}\label{identity_b}
	\sum_{l=1}^m\eta_l\hat b_l=\langle \eta,\hat b\rangle=\langle \eta,(\tilde A \tilde A^t)^{-\frac12}b\rangle=\langle \eta,(\tilde A \tilde A^t)^{-\frac12}A\frac1w\rangle=\sum_{l=1}^m\sum_{j=1}^n\eta_l\hat A_{lj}=\sum_{j=1}^n\langle\eta,\hat A_{\bullet j}\rangle,
	\end{align}
	and $\tilde A,\hat A,\hat b$ have been defined in \eqref{def:Atilde} and \eqref{def:Ahat}.
	Recall that for a random variable $X$ with characteristic function analytic in a neighborhood of 0, we have for $s$ in that neighborhood
	\begin{align}
	    \E \exp({\i sX})=\exp\lb\sum_{k=1}^\infty\frac{\i^ks^k\k_k}{k!}\rb,
	\end{align}
 where $\k_k=\k_k^\n,\,k=1,\dots$ are the cumulants of the random variable $X$. We will apply this  to
 \begin{align}
     X:=\sum_{j=1}^n X_j\lb t\hat\l_j+\langle\eta,\hat A_{\bullet j}\rangle\rb-\sum_{j=1}^n \lb t\hat\l_j+\langle\eta,\hat A_{\bullet j}\rangle\rb.
 \end{align}
 The $X_j$ are i.i.d.~exponentially distributed with parameter 1. By our assumptions, $t\hat\l_j+\langle\eta,\hat A_{\bullet j}\rangle\to0$ uniformly in $j$ as $n\to\infty$ and $\s$ is bounded away from 0, hence $s=\s^{-1}$ is in the region of analyticity for $n$ large enough. Let us assume for later use that $\lvv t\hat\l_j+\langle\eta,\hat A_{\bullet j}\rangle\rvv\leq1/2$ for $j=1,\dots,n$. 
We then get for any fixed $\eta$ and $n$ large enough
	\begin{align}
		&\E_{P_{n,\mb 1}} \exp\lb\i\s^{-1} \lb\sum_{j=1}^n X_j\lb t\hat\l_j+\langle\eta,\hat A_{\bullet j}\rangle\rb-\sum_{j=1}^n \lb t\hat\l_j+\langle\eta,\hat A_{\bullet j}\rangle\rb\rb\rb=\exp\lb\sum_{k=1}^\infty\frac{\i^k\k_k}{\s^kk!}\rb,
	\end{align}
	where now $\k_k=\k_k^\n,\,k=1,\dots$ are the cumulants of the random variable
	\begin{align}
		\sum_{j=1}^n X_j\lb t\hat\l_j+\langle\eta,\hat A_{\bullet j}\rangle\rb-\sum_{j=1}^n \lb t\hat\l_j+\langle\eta,\hat A_{\bullet j}\rangle\rb.
	\end{align}
Since the $X_j$ are i.i.d., and for $k\geq2$ the cumulant $\k_k$ is invariant under shifts, homogeneous of degree $k$ and the $k$th cumulant of an exponentially distributed random variable with mean 1 is $(k-1)!$, we have the mean $\k_1=0$ and for $k\geq 2$
\begin{align}
	\k_k=(k-1)!\sum_{j=1}^n\lb t\hat\l_j+\langle\eta,\hat A_{\bullet j}\rangle\rb^k.\label{cumulant}
\end{align}
Straightforward algebra shows that for the variance we get
\begin{align}
	&\k_2=\operatorname{Var}\lb \sum_{j=1}^n X_j\lb t\hat\l_j+\langle\eta,\hat A_{\bullet j}\rangle\rb-\sum_{j=1}^n \lb t\hat\l_j+\langle\eta,\hat A_{\bullet j}\rangle\rb\rb\\
	&=t^2\|\hat\l\|^2+\langle \eta,\hat A\hat A^t\eta\rangle+2t\langle \eta,\hat A\hat\l\rangle=t^2\|\hat\l\|^2+\|\eta\|^2+2t\langle \eta,\hat A\hat\l\rangle.
\end{align}
Note that by 
\begin{align}
	\lv \langle \eta,\hat A\hat\l\rangle\leq\|\eta\|\|\hat A\hat\l\|\leq\|\eta\|\|\hat\l\|,
\end{align} 
we in particular have $\sup_n\k_2<\infty$. As $\lvv t\hat\l_j+\langle\eta,\hat A_{\bullet j}\rangle\rvv \leq 1/2$, we find for $k\geq 3$
\begin{align}
    \lv\kappa_k\rv\leq (k-1)!\sum_{j=1}^n\lvv t\hat\l_j+\langle\eta,\hat A_{\bullet j}\rangle\rvv^k
    &\leq (k-1)!(1/2)^{k-3}\sup_j \lvv t\hat\l_j+\langle\eta,\hat A_{\bullet j}\rangle\rvv\sum_{j=1}^n\lvv t\hat\l_j+\langle\eta,\hat A_{\bullet j}\rangle\rvv^2\\
    &=(k-1)!(1/2)^{k-3}o(1),\label{reduction_cumulants}
\end{align}
where the $o(1)$ term is independent of $k$.
We conclude from \eqref{reduction_cumulants} that 
\begin{align}
  \sum_{k=3}^\infty\frac{\i^k\k_k}{\s^kk!}=o(1)
\end{align}
    as $n\to\infty$ and 
hence we obtain for any fixed $\eta$
\begin{align}
	&\E_{P_{n,\mb 1}} \exp\lb{\i \s^{-1}t \hat S_n+\i\s^{-1}\sum_{l=1}^m\eta_l\lb \hat A_{l\bullet}X-  \hat b_l\rb}\rb\\
	&=\exp\lb-\frac1{2\s^{-2}}\lb t^2\|\hat\l\|^2+\|\eta\|^2+2t\langle \eta,\hat A\hat\l\rangle\rb\rb+o(1)\label{exponential-bound}
\end{align}
as $n\to\infty$. To insert this into the $\eta$-integral in the numerator of \eqref{Bartlett2} (with the change of variables $\eta\to\s^{-1}\eta$), we will show that the l.h.s.~of \eqref{exponential-bound} has an integrable dominating function.
To this end, by Proposition \ref{prop:examplebis} (1) there are index sets $I_1,\dots,I_{m+1}\subset \{1,\dots,n\}$ pairwise disjoint and non-empty such that \eqref{det-condition} holds. As the $X_j$ are i.i.d.~exponential random variables with parameter 1, we have 
	\begin{align}
		&\lvv\E_{P_{n,\mb 1}} \exp\lb{\i\s^{-1} t \hat S_n+\i\s^{-1}\sum_{l=1}^m\eta_l\lb \hat A_{l\bullet}X-   \hat b_l\rb}\rb\rvv
		=\lb\prod_{j=1}^n \lb1+\s^{-2}\lb t\hat\l_j+\langle\eta,\hat A_{\bullet j}\rangle\rb^2\rb\rb^{-\frac12}\\
		&\leq \prod_{l'=1}^{m+1}\lb\prod_{j\in I_{l'}} \lb1+\s^{-2}\lb t\hat\l_j+\langle\eta,\hat A_{\bullet j}\rangle\rb^2\rb\rb^{-\frac12}.\label{bound_characteristic_function}
	\end{align}
Using the crude inequality
\begin{align}
	\prod_{j=1}^n(1+x_j)\geq1+\sum_{j=1}^n x_j,
\end{align}
valid for any $x_1,\dots,x_n\geq0$, on each of the products over elements in $I_{l'}$, we obtain 
\begin{align}
	&\lvv\E_{P_{n,\mb 1}} \exp\lb{\i\s^{-1} t	 \hat S_n+\i\s^{-1}\sum_{l=1}^m\eta_l\lb \hat A_{l\bullet}X-   \hat b_l\rb}\rb\rvv\leq  \prod_{l'=1}^{m+1} \lb 1+\s^{-2}\sum_{j\in I_{l'}}\lb t\hat\l_j+\langle\eta,\hat A_{\bullet j}\rangle\rb^2\rb^{-\frac12}\\
	&=\prod_{l'=1}^{m+1} \lb 1+\s^{-2}\left\langle\eta,\hat A_{I_{l'}}\hat A_{I_{l'}}^t\eta\right\rangle +\s^{-2}t^2\|\hat\l\|^2+\s^{-2}2t\left\langle \eta,\hat A_{I_{l'}}\hat\l\right\rangle\rb^{-\frac12}.\label{uniform_boundedness}
\end{align}
By \eqref{det-condition} we have $\det(\hat A_{I_{l'}}\hat A_{I_{l'}}^t)\geq C>0$ for some $C$ not depending on $n$. This is equivalent to the smallest eigenvalue of $\hat A_{I_{l'}}\hat A_{I_{l'}}^t$ being bounded below by some $n$-independent $c>0$. We thus obtain
\begin{align}
	\left\langle\eta,\hat A_{I_{l'}}\hat A_{I_{l'}}^t\eta\right\rangle\geq c\sum_{l=1}^m\eta_l^2
\end{align}
and consequently for $\|\eta\|$ large enough
\begin{align}
	&\lvv\E_{P_{n,\mb 1}} \exp\lb{\i\s^{-2} t	 \hat S_n+\i\s^{-2}\sum_{l=1}^m\eta_l\lb \hat A_{l\bullet}X-   \hat b_l\rb}\rb\rvv
	\leq  \prod_{l'=1}^{m+1} \lb 1+c\s^{-2}\|\eta\|^2-2\s^{-2}\lv t\rv\lvv\left\langle \eta,\hat A_{I_{l'}}\hat\l\right\rangle\rvv\rb^{-\frac12}\\
	&\leq\prod_{l'=1}^{m+1} \lb 1+c\s^{-2}\|\eta\|^2-2\s^{-2}\lv t\rv\|\eta\|\|\hat\l\|\rb^{-\frac12},
\end{align}
where we used $\|\hat A_{I_l',l\bullet}\|\leq \|\hat A_{l\bullet}\|=1$. Since $\|\hat\l\|$ is bounded in $n$, we have for some $C>0$ independent of $n$ and $\|\eta\|>C$
\begin{align}
	&\lvv\E_{P_{n,\mb 1}} \exp\lb{\i\s^{-1} t	 \hat S_n+\i\s^{-1}\sum_{l=1}^{m+1}\eta_l\lb \hat A_{l\bullet}X-   \hat b_l\rb}\rb\rvv
	\leq\prod_{l'=1}^{m+1} \lb 1+\frac{c}{2\s^2}\|\eta\|^2\rb^{-\frac12}=\lb 1+\frac{c}{2\s^2}\|\eta\|^2\rb^{-\frac {m+1}2}.
\end{align}
For $\|\eta\|\leq C$ we simply use the crude bound 
\begin{align}
	&\lvv\E_{P_{n,\mb 1}} \exp\lb{\i \s^{-1}t	 \hat S_n+\i\s^{-1}\sum_{l=1}^m\eta_l\lb \hat A_{l\bullet}X-   \hat b_l\rb}\rb\rvv
	\leq 1.
\end{align}
This provides an integrable function as 
\begin{align}
	\int_{\|\eta\|>C}\lb 1+\frac{c}{2\s^{2}}\|\eta\|^2\rb^{-\frac {m+1}2}d\eta\leq 2\pi^{m-1}\int_0^\infty r^{m-1}\lb 1+\frac{c}{2\s^2}r^2\rb^{-\frac {m+1}2}dr<\infty.\label{end_bound}
\end{align}
This proves, as $n\to\infty$,
\begin{align}
	&\E_{P_n}e^{\i \s^{-1}t S_n^*}=\frac{\int_{\R^m}\exp\lb-\frac1{2\s^2}\lb t^2\|\hat\l\|^2+\|\eta\|^2+2t\langle \eta,\hat A\hat\l\rangle\rb\rb d\eta}{\int_{\R^m}\exp\lb-\frac1{2\s^2} \|\eta\|^2\rb d\eta}+o(1)\\
	&=\exp\lb-\frac{t^2}{2\s^2}\lb\|\hat\l\|^2-\|\hat A\hat\l\|^2\rb\rb+o(1).\label{equ:result}
\end{align}
It remains to show that 
\begin{align}
    \|\hat\l\|^2-\|\hat A\hat\l\|^2=\|P_{\operatorname{ker(\hat A)}}\hat\l\|^2=\s^2,
\end{align}
where $P_{\operatorname{ker(\hat A)}}$ is the orthogonal projection to the kernel of $\hat A$. This however follows immediately from $\hat A\hat A^t=I$ as
\begin{align}
    \|\hat\l\|^2-\|\hat A\hat\l\|^2=\|(I-\hat A^t\hat A)\hat\l\|^2=\|(I-P_{\textup{Im}(\hat A^t)})\hat\l\|^2,
\end{align}
where $P_{\textup{Im}(\hat A^t)}$ is the orthogonal projection to the image of $\hat A^t$. This finishes the proof for $S_n^*$. \\

To deduce the result for $S_n$ from the one for $S_n^*$, it suffices to show that 
\begin{align}\label{equ:toshow}
    \E_{P_n}\s^{-1}\sum_{j=1}^n\l_jX_j-\E_{P_{n,w}}\s^{-1}\sum_{j=1}^n\l_jX_j=\E_{P_n}\s^{-1}S_n^*=o(1)
\end{align}
as $n\to\infty$.
From
\begin{align}
	&\E_{P_n} \s^{-1}S_n^*
	=-\i\frac{d}{dt}\E_{P_n} e^{\i \s^{-1}tS_n^*}|_{t=0}
\end{align}
 and \eqref{equ:result}
we see that \eqref{equ:toshow} will follow if we can interchange taking the limit $n\to\infty$ and differentiating at $t=0$. To justify this we will show that $t\mapsto\E_{P_n} e^{\i \s^{-1}tS_n^*}$, which is an entire function thanks to the boundedness of $S_n^*$, converges uniformly in a complex neighborhood of $t=0$ to the r.h.s.~of \eqref{equ:result} as $n\to\infty$. We first show that the representation \eqref{Bartlett2'} continues to hold for $t\in U$, where $U\subset \C$ is an open neighborhood of $0$ to be chosen small enough. For this it suffices to show analyticity of the r.h.s.~of \eqref{Bartlett2'}, which can be seen using Morera's theorem and Fubini's theorem. We also  remark that thanks to the $X_j$ being i.i.d.~exponential, the moment generating function of $S_n^*$ exists and hence, using that $\s$ is bounded away from 0 and assuming without loss of generality that all $\lv\hat\l_j\rv\leq1$, we get for some $C>0$ independent of $n$ and all $n$ large enough
\begin{align}\label{equ:bound_analytic}
    \lvv \E_{P_{n,\mb 1}}e^{\i\s^{-1} t \hat S_n+\i\s^{-1}\sum_{l=1}^m\eta_l\lb \hat A_{l\bullet}X-   \hat b_l\rb}\rvv\leq \E_{P_{n,\mb 1}}e^{-\Im(t)\s^{-1} t \hat S_n}\leq  C,
\end{align}
uniformly for $t\in U$ chosen small enough. We conclude that the r.h.s.~of \eqref{Bartlett2'} is analytic for $\lv\Im(t)\rv$ small enough and as the l.h.s.~is entire and both sides coincide for $t\in\R$, by the identity theorem representation \eqref{Bartlett2'} extends to complex $t$ for $\lv\Im(t)\rv$ small enough. Now that we know that both sides of \eqref{Bartlett2'} are analytic in $t$ around 0, we will use Vitali's convergence theorem, a consequence of Montel's theorem, to deduce convergence of the l.h.s. of \eqref{Bartlett2'} to the r.h.s.~of \eqref{equ:result}, uniformly in $t$ from a small complex neighborhood of $t$. More precisely, Vitali's Convergence Theorem \cite{Titchmarsh} states that if a sequence of analytic functions is uniformly bounded on a complex domain and converges pointwise on a subset of that domain that has a limit point, then in fact the convergence is uniform on the whole domain. In our application the subset of the domain is the intersection of the complex domain with $\R$, and convergence on it is provided by \eqref{equ:result}. To prove the required uniform boundedness we follow the steps from \eqref{bound_characteristic_function} to \eqref{end_bound}. The starting equation in \eqref{bound_characteristic_function} has now for complex $t$ to be replaced by 
	\begin{align}
		&\lvv\E_{P_{n,\mb 1}} \exp\lb{\i\s^{-1} t \hat S_n+\i\s^{-1}\sum_{l=1}^m\eta_l\lb \hat A_{l\bullet}X-   \hat b_l\rb}\rb\rvv\\
		&=\lb\prod_{j=1}^n \lb1-2\s^{-1}\Im(t)\hat\l_j+\s^{-2}\lvv t\hat\l_j+\langle\eta,\hat A_{\bullet j}\rangle\rvv^2\rb\rb^{-\frac12},
	\end{align}
but as $t$ is bounded, the decay in $\eta$ is the same as in \eqref{bound_characteristic_function} and the analog of \eqref{end_bound} provides a bound uniform in $n$ and $t$. 
This allows to end the proof.
\end{proof}

\begin{proof}[Proof of Theorem \ref{thrm_marginal}]
This proof follows the strategy of the proof of Theorem \ref{thrm1}. We will look at the joint characteristic function $\E_{P_n}\exp(\i\sum_{i=1}^k t_{j_i}w_{j_i}X_{j_i})$,
where $t_{j_i}\in\R,\ i=1,\dots,k$ are arbitrary.
In analogy to \eqref{Bartlett2} we get
\begin{align}
	&\E_{P_n}e^{\i \sum_{i=1}^k t_{j_i}w_{j_i}X_{j_i}}
	=\frac{\int_{\R^m}\E_{P_{n,\mb 1}} \exp\lb{\i  \sum_{i=1}^k t_{j_i}X_{j_i}+\i\sum_{l=1}^m\eta_l\lb \hat A_{l\bullet}X-   \hat b_l\rb}\rb d\eta}{\int_{\R^m}\E_{P_{n,\mb 1}}\exp\lb{\i\sum_{l=1}^m\eta_l\lb \hat A_{l\bullet}X-   \hat b_l\rb}\rb d\eta}\label{Bartlett3}.
	\end{align}
	We now have with \eqref{identity_b}
	\begin{align}
	&\E_{P_{n,\mb 1}} \exp\lb{\i \sum_{i=1}^k t_{j_i}X_{j_i}+\i\sum_{l=1}^m\eta_l\lb \hat A_{l\bullet}X-   \hat b_l\rb}\rb\\
	&=\E_{P_{n,\mb 1}} \exp\lb\i \sum_{i=1}^k (t_{j_i}+\langle\eta,\hat A_{\bullet j_i}\rangle)X_{j_i}\rb\cdot\E_{P_{n,\mb 1}} \exp\lb\i \sum_{j\not=j_i,i=1,\dots,k} \langle\eta,\hat A_{\bullet j}\rangle)(X_j-1)\rb.
	\end{align}
As $\|\hat A\|_{\max}=o(1)$, we conclude that by Slutsky's lemma the first factor converges as $n\to\infty$ to the joint characteristic function of $k$ i.i.d.~exponential random variables with parameter 1. We may now use the same steps as in the proof of Theorem \ref{thrm1} to justify interchanging limit and integration in the numerator, as well as showing that
\begin{align}
    \lim_{n\to\infty}\frac{\int_{\R^m}\E_{P_{n,\mb 1}} \exp\lb{\i  \sum_{j\not=j_i,i=1,\dots,k} \langle\eta,\hat A_{\bullet j}\rangle)(X_j-1)}\rb d\eta}{\int_{\R^m}\E_{P_{n,\mb 1}}\exp\lb\i\sum_{j=1}^n \langle\eta,\hat A_{\bullet j}\rangle)(X_j-1)\rb d\eta}
    =\frac{\int_{\R^m}\exp\lb-\|\eta\|^2\rb d\eta}{\int_{\R^m}\exp\lb-\|\eta\|^2\rb d\eta}=1.
\end{align}
We note that here we use Property A in the form of \eqref{det-condition} with $K=m+1+k$ for the numerator, instead of $K=m+1$ used in the proof of Theorem \ref{thrm1}.
\end{proof}

{\bf Acknowledgments} Our warmest thanks to Franck Barthe for his bibliographical insights into the Brascamp-Lieb inequality. Support from the ANR-3IA Artificial and Natural Intelligence Toulouse Institute is gratefully acknowledged. MV is grateful for partial support from the Irish Research Council through a Ulysses Award. 
\section*{Appendix: A simple Python program to determine the probabilistic barycenter}
Here we present a short Python code to demonstrate the practical possibility of calculating the entropy center of the polytope $K_n$. To do this, we implement the dual optimization problem \eqref{Edual}.
The target function is named \textmtt{dualentrop} in the code. In the current example, we are working in $\R^{30}$ with $2$ constraints. On the one hand, the first constraint requires the vector to be a probability measure on $\{1,\cdots,30\}$. On the other hand, the second constraint is the hyperplane
$${x\in\R^{30} : x_1+\cdots +x_5+2x_6+\cdots+ 2x_{10}=0.8}.$$
The output is the entropy center whose components are block-wise constant. Up to the numerical precision, it 
almost satisfies the $2$ constraints as expected.

\includegraphics[width=15cm]{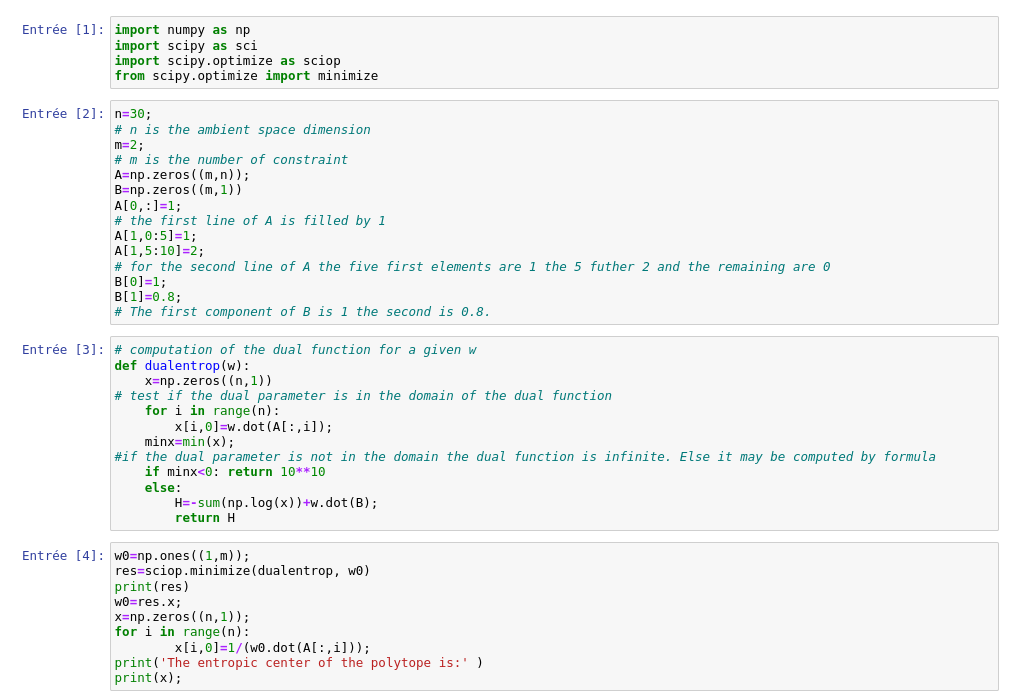}
\newpage
\includegraphics[width=15cm]{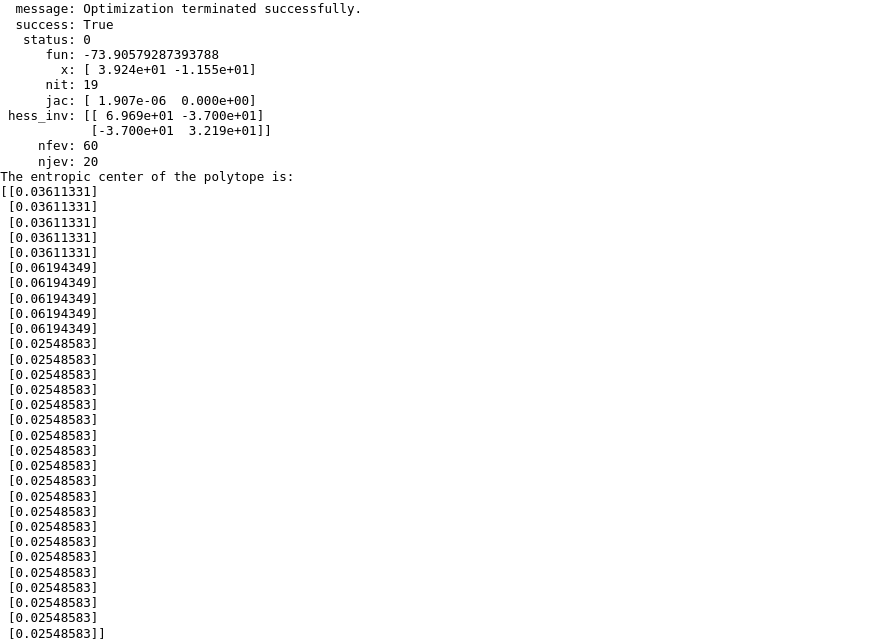}
\printbibliography
\end{document}